\documentclass[leqno,12pt]{amsart} 
\usepackage{amssymb}
\usepackage{amsmath}
\usepackage{amsthm}
\usepackage{amscd}
\usepackage{graphicx}
\usepackage[dvips]{color}
\theoremstyle{plain} 
\newtheorem{theorem}{\indent\sc Theorem}[section]
\newtheorem{lemma}[theorem]{\indent\sc Lemma}
\newtheorem{corollary}[theorem]{\indent\sc Corollary}
\newtheorem{proposition}[theorem]{\indent\sc Proposition}

\theoremstyle{definition} 
\newtheorem{definition}[theorem]{\indent\sc Definition}
\newtheorem{remark}[theorem]{\indent\sc Remark}
\newtheorem{example}[theorem]{\indent\sc Example}

\newtheorem{observation}[theorem]{\indent\sc Observation}
\begin{document}

\title[Minimal singular metrics of a line bundle]
{Minimal singular metrics of a line bundle admitting no Zariski decomposition} 

\author[T. Koike]{Takayuki Koike} 

\subjclass[2010]{ 
Primary 32J25; Secondary 32J27, 14C20. 
}
%
\keywords{ 
Minimal singular metrics, Zariski decompositions, Nakayama example, 
Kiselman numbers, Lelong numbers, non-nef loci, 
multiplier ideal sheaves. 
}
\address{
Mathematical Institute \endgraf
The University of Tokyo \endgraf
3-8-1 Komaba, Tokyo \endgraf
Japan
}
\email{tkoike@ms.u-tokyo.ac.jp}

\maketitle

\begin{abstract}
We give a concrete expression of a minimal singular metric on a big line bundle 
on a compact K\"ahler manifold which is the total space of a toric bundle over a complex torus. 
In this class of manifolds, Nakayama constructed examples which have line bundles admitting no Zariski decomposition 
even after modifications. 
As an application, we discuss the Zariski closedness of non-nef loci. 
\end{abstract}

\section{Introduction}
We consider the positivity of a big holomorphic line bundle over a compact K\"ahler complex manifold. 
Especially, we are interested in the information related to the obstruction to the nef-ness of the line bundle. 
Our main result is the explicit construction of a minimal singular metric, or a singular hermitian metric on $L$ with minimal singularities, of a big line bundle $L$ 
when the manifold $X$ is the total space of a smooth projective toric bundle over a complex torus (Theorem \ref{main_theorem}). 
\par
In order to state our main theorem in general form, we have to define some terminology. 
So in this section, we introduce our result only when $(X, L)$ is a Nakayama example (\cite[IV \S 2.6]{N}), 
which is one of the most important examples when we study the obstruction to the nef-ness of the line bundle, since it admits no Zariski decomposition even after modifications. 
Let $E_1$ be a sufficiently general smooth elliptic curve such as $\mathbb{C}/(\mathbb{Z}+(\pi+\sqrt{-1})\mathbb{Z})$, $E_2$ a copy of $E_1$, 
and $z_j$ a coordinate of $E_j$ for $j=1, 2$. 
Let us fix an integer $a>1$, points $p_1\in E_1, p_2\in E_2$, and define the three line bundles $L_j (j=0, 1, 2)$ over $V=E_1\times E_2$ by
\begin{eqnarray}
L_0&=&\mathcal{O}_V(2F_1-4F_2+2\Delta), \nonumber \\
L_1&=&\mathcal{O}_V((a-1)F_1+(a-1)F_2+(a+2)\Delta),  \nonumber \\
L_2&=&\mathcal{O}_V((a+3)F_1+(a-3)F_2+a\Delta), \nonumber
\end{eqnarray}
where $F_1$ stands for the prime divisor $\{p_1\}\times E_2\subset V$, $F_2$ stands for the prime divisor $E_1\times \{p_2\}\subset V$, and
$\Delta$ stands for the prime divisor $\{(x, y)\in E\times E\mid x=y\}$.
Then there exists a hermitian metric $h_j$ over $L_j$ whose curvature tensor $\Theta_{h_j}\in c_1(L_j)$ is a harmonic form and
each $h_j$ can be denoted as $h_{j}(\xi, \eta)_{(z_1, z_2)}=e^{-\varphi_j(z_1, z_2)}\xi\overline{\eta}$, where
\begin{eqnarray}
\varphi_0(z_1, z_2)&=&
(z_1, z_2)
\left( \begin{array}{cc}
4 & -2 \\
-2 & -2
\end{array} \right)
\overline{
\left(\hskip-2mm \begin{array}{c}
z_1\\
z_2
\end{array} \hskip-2mm\right)
}\nonumber \\
\varphi_1(z_1, z_2)&=&
(z_1, z_2)
\left( \begin{array}{cc}
2a+1 & -(a+2) \\
-(a+2) & 2a+1
\end{array} \right)
\overline{
\left( \hskip-2mm\begin{array}{c}
z_1\\
z_2
\end{array} \hskip-2mm\right)
}\nonumber \\
\varphi_2(z_1, z_2)&=&
(z_1, z_2)
\left( \begin{array}{cc}
2a+3 & -a \\
-a & 2a-3
\end{array} \right)
\overline{
\left( \hskip-2mm\begin{array}{c}
z_1\\
z_2
\end{array} \hskip-2mm\right)}, \nonumber
\end{eqnarray}
on each small open subset $U$ of $V$ with appropriate local trivialization $s^j$ of $L_j$ on $U$. 
Let us define the variety $X$ as the total space of a $\mathbb{P}^2$-bundle $\pi\colon \mathbb{P}(L_0\oplus L_1\oplus L_2)\to V$ over $V$ and $L=\mathcal{O}_{\mathbb{P}(L_0\oplus L_1\oplus L_2)}(1)$. 
Let $U$ be a sufficiently small open set of $V$.  
We use the function 
\begin{eqnarray}
([x_0; x_1; x_2], z_1, z_2)&\mapsto&[x_0s_0(z_1, z_2); x_1s_1(z_1, z_2); x_2s_2(z_1, z_2)]\nonumber \\
&\in&(\mathbb{C}s^0(z_1, z_2)\oplus\mathbb{C}s^1(z_1, z_2)\oplus\mathbb{C}s^2(z_1, z_2))^*/\mathbb{C}^*=\pi^{-1}(z_1, z_2)\nonumber
\end{eqnarray}
as a coordinates system on $\pi^{-1}(U)$, where $s_j$ is a dual section of $s^j$. 
Using these coordinates, our main result applied to this example can be stated as follows: 

\begin{theorem}
Let $(X, L)$ be the above example, which is introduced by Nakayama \cite{N} and admits no Zariski decomposition even after modifications. 
There is a minimal singular metric $h_{\rm min}$ on $L$ 
whose local weight function $\psi$ is continuous on $X\setminus\mathbb{P}(L_0)$ and 
is written as 
\[\psi=\log\max_{(\alpha, \beta)\in H}\,(|x_1|^{2\alpha}\cdot|x_2|^{2\beta})+O(1)\]
at each point in $\mathbb{P}(L_0)$ with local coordinates $(x_1, x_2, z_1, z_2)=([1; x_1; x_2], z_1, z_2)$, 
where $H=\{(\alpha, \beta)\in\mathbb{R}^2\mid\alpha, \beta\geq 0,\ a^2(\alpha+\beta)^2=(1-\alpha)^2+(1-\beta)^2\}$. 
\end{theorem}

This expression enables us to compute the multiplier ideal sheaf $\mathcal{J}(h_{\rm min}^t)$ 
for each positive number $t$, whose stalk at $x_0\in X$ is defined by 
\[\mathcal{J}(h_{\rm min}^t)_{x_0}=\{f\in\mathcal{O}_{X, x}\mid |f|^2e^{-t\varphi_{\rm min}}\text{ \ is\ integrable\ around\ }x_0\}, \]
where $\varphi_{\rm min}$ is the local weight function of $h_{\rm min}$ around $x_0$. 

\begin{corollary}
$\mathcal{J}(h_{\rm min}^t)$ is trivial at any point in $X\setminus\mathbb{P}(L_0)$. 
For a point $x_0\in\mathbb{P}(L_0)$, the stalk $\mathcal{J}(h_{\rm min})_{x_0}$ of the multiplier ideal sheaf is the ideal of $\mathcal{O}_{X, x_0}$ which is generated by the polynomials
\[\{x_1^px_2^q\mid (p+1, q+1)\in {\rm Int}(S_t)\cap\mathbb{Z}^2\}, \] 
where we denote by $S_t$ the set $\{(t\alpha, t\beta)\in\mathbb{R}^2\mid\alpha, \beta\geq 0,\ a^2(\alpha+\beta)^2\geq(1-\alpha)^2+(1-\beta)^2\}$ (
For the shape of $S_t$ in this case, see Figure \ref{example2}). 
\end{corollary}

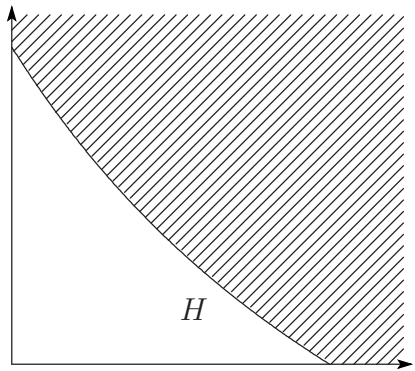
\begin{figure}[htbp]
  \begin{center}
\unitlength 0.1in
\begin{picture}( 21.1700, 20.3700)(  3.7300,-20.3700)
%
{\color[named]{Black}{%
\special{pn 8}%
\special{pa 400 2010}%
\special{pa 2490 2010}%
\special{fp}%
\special{sh 1}%
\special{pa 2490 2010}%
\special{pa 2424 1990}%
\special{pa 2438 2010}%
\special{pa 2424 2030}%
\special{pa 2490 2010}%
\special{fp}%
}}%
%
{\color[named]{Black}{%
\special{pn 8}%
\special{pa 400 2010}%
\special{pa 400 140}%
\special{fp}%
\special{sh 1}%
\special{pa 400 140}%
\special{pa 380 208}%
\special{pa 400 194}%
\special{pa 420 208}%
\special{pa 400 140}%
\special{fp}%
}}%
{\color[named]{Black}{%
\special{pn 8}%
\special{pa 400 350}%
\special{pa 420 382}%
\special{pa 426 388}%
\special{pa 440 412}%
\special{pa 446 420}%
\special{pa 456 436}%
\special{pa 460 442}%
\special{pa 470 458}%
\special{pa 476 466}%
\special{pa 480 474}%
\special{pa 486 480}%
\special{pa 490 488}%
\special{pa 496 496}%
\special{pa 500 504}%
\special{pa 510 518}%
\special{pa 516 526}%
\special{pa 526 540}%
\special{pa 530 548}%
\special{pa 546 568}%
\special{pa 550 576}%
\special{pa 626 682}%
\special{pa 630 688}%
\special{pa 646 708}%
\special{pa 650 714}%
\special{pa 666 736}%
\special{pa 670 742}%
\special{pa 676 748}%
\special{pa 680 754}%
\special{pa 690 768}%
\special{pa 696 774}%
\special{pa 700 782}%
\special{pa 706 788}%
\special{pa 710 794}%
\special{pa 716 800}%
\special{pa 720 808}%
\special{pa 730 820}%
\special{pa 736 826}%
\special{pa 746 838}%
\special{pa 750 846}%
\special{pa 760 858}%
\special{pa 766 864}%
\special{pa 786 888}%
\special{pa 790 896}%
\special{pa 866 986}%
\special{pa 870 990}%
\special{pa 896 1020}%
\special{pa 900 1026}%
\special{pa 910 1038}%
\special{pa 916 1042}%
\special{pa 926 1054}%
\special{pa 930 1060}%
\special{pa 940 1072}%
\special{pa 946 1076}%
\special{pa 950 1082}%
\special{pa 956 1088}%
\special{pa 960 1094}%
\special{pa 966 1098}%
\special{pa 970 1104}%
\special{pa 976 1110}%
\special{pa 980 1116}%
\special{pa 986 1120}%
\special{pa 990 1126}%
\special{pa 996 1132}%
\special{pa 1000 1138}%
\special{pa 1026 1164}%
\special{pa 1030 1170}%
\special{pa 1096 1236}%
\special{pa 1100 1242}%
\special{pa 1200 1342}%
\special{pa 1206 1346}%
\special{pa 1220 1360}%
\special{pa 1226 1364}%
\special{pa 1240 1380}%
\special{pa 1246 1384}%
\special{pa 1270 1408}%
\special{pa 1276 1412}%
\special{pa 1280 1416}%
\special{pa 1286 1420}%
\special{pa 1310 1444}%
\special{pa 1316 1448}%
\special{pa 1330 1462}%
\special{pa 1336 1466}%
\special{pa 1346 1476}%
\special{pa 1356 1484}%
\special{pa 1360 1488}%
\special{pa 1366 1492}%
\special{pa 1370 1498}%
\special{pa 1380 1506}%
\special{pa 1386 1510}%
\special{pa 1396 1518}%
\special{pa 1400 1524}%
\special{pa 1410 1532}%
\special{pa 1416 1536}%
\special{pa 1430 1548}%
\special{pa 1436 1554}%
\special{pa 1456 1570}%
\special{pa 1460 1574}%
\special{pa 1490 1598}%
\special{pa 1496 1604}%
\special{pa 1596 1684}%
\special{pa 1600 1686}%
\special{pa 1630 1710}%
\special{pa 1636 1714}%
\special{pa 1656 1730}%
\special{pa 1660 1732}%
\special{pa 1676 1744}%
\special{pa 1680 1748}%
\special{pa 1696 1760}%
\special{pa 1700 1762}%
\special{pa 1710 1770}%
\special{pa 1716 1774}%
\special{pa 1726 1782}%
\special{pa 1730 1784}%
\special{pa 1740 1792}%
\special{pa 1746 1796}%
\special{pa 1750 1800}%
\special{pa 1756 1802}%
\special{pa 1766 1810}%
\special{pa 1770 1814}%
\special{pa 1776 1818}%
\special{pa 1780 1820}%
\special{pa 1790 1828}%
\special{pa 1796 1832}%
\special{pa 1800 1836}%
\special{pa 1806 1838}%
\special{pa 1810 1842}%
\special{pa 1816 1846}%
\special{pa 1820 1850}%
\special{pa 1826 1852}%
\special{pa 1830 1856}%
\special{pa 1836 1860}%
\special{pa 1840 1864}%
\special{pa 1846 1866}%
\special{pa 1850 1870}%
\special{pa 1860 1876}%
\special{pa 1866 1880}%
\special{pa 1870 1884}%
\special{pa 1876 1888}%
\special{pa 1880 1890}%
\special{pa 1886 1894}%
\special{pa 1896 1900}%
\special{pa 1900 1904}%
\special{pa 1910 1910}%
\special{pa 1916 1914}%
\special{pa 1926 1920}%
\special{pa 1930 1924}%
\special{pa 1940 1930}%
\special{pa 1946 1934}%
\special{pa 1956 1940}%
\special{pa 1960 1944}%
\special{pa 1976 1954}%
\special{pa 1980 1958}%
\special{pa 1996 1966}%
\special{pa 2000 1970}%
\special{pa 2016 1980}%
\special{pa 2020 1984}%
\special{pa 2046 1998}%
\special{pa 2050 2002}%
\special{pa 2060 2008}%
\special{pa 2064 2010}%
\special{fp}%
}}%
\put(12.8000,-17.8000){\makebox(0,0)[lb]{$H$}}%
%
{\color[named]{Black}{%
\special{pn 8}%
\special{pa 2450 910}%
\special{pa 1650 1710}%
\special{fp}%
\special{pa 2450 970}%
\special{pa 1680 1740}%
\special{fp}%
\special{pa 2450 1030}%
\special{pa 1710 1770}%
\special{fp}%
\special{pa 2450 1090}%
\special{pa 1750 1790}%
\special{fp}%
\special{pa 2450 1150}%
\special{pa 1780 1820}%
\special{fp}%
\special{pa 2450 1210}%
\special{pa 1820 1840}%
\special{fp}%
\special{pa 2450 1270}%
\special{pa 1850 1870}%
\special{fp}%
\special{pa 2450 1330}%
\special{pa 1890 1890}%
\special{fp}%
\special{pa 2450 1390}%
\special{pa 1930 1910}%
\special{fp}%
\special{pa 2450 1450}%
\special{pa 1960 1940}%
\special{fp}%
\special{pa 2450 1510}%
\special{pa 2000 1960}%
\special{fp}%
\special{pa 2450 1570}%
\special{pa 2040 1980}%
\special{fp}%
\special{pa 2450 1630}%
\special{pa 2070 2010}%
\special{fp}%
\special{pa 2450 1690}%
\special{pa 2130 2010}%
\special{fp}%
\special{pa 2450 1750}%
\special{pa 2190 2010}%
\special{fp}%
\special{pa 2450 1810}%
\special{pa 2250 2010}%
\special{fp}%
\special{pa 2450 1870}%
\special{pa 2310 2010}%
\special{fp}%
\special{pa 2450 1930}%
\special{pa 2370 2010}%
\special{fp}%
\special{pa 2450 850}%
\special{pa 1610 1690}%
\special{fp}%
\special{pa 2450 790}%
\special{pa 1580 1660}%
\special{fp}%
\special{pa 2450 730}%
\special{pa 1550 1630}%
\special{fp}%
\special{pa 2450 670}%
\special{pa 1510 1610}%
\special{fp}%
\special{pa 2450 610}%
\special{pa 1480 1580}%
\special{fp}%
\special{pa 2450 550}%
\special{pa 1450 1550}%
\special{fp}%
\special{pa 2450 490}%
\special{pa 1410 1530}%
\special{fp}%
\special{pa 2450 430}%
\special{pa 1380 1500}%
\special{fp}%
\special{pa 2450 370}%
\special{pa 1350 1470}%
\special{fp}%
\special{pa 2450 310}%
\special{pa 1320 1440}%
\special{fp}%
\special{pa 2450 250}%
\special{pa 1290 1410}%
\special{fp}%
\special{pa 2450 190}%
\special{pa 1260 1380}%
\special{fp}%
}}%
%
{\color[named]{Black}{%
\special{pn 8}%
\special{pa 2400 180}%
\special{pa 1220 1360}%
\special{fp}%
\special{pa 2340 180}%
\special{pa 1190 1330}%
\special{fp}%
\special{pa 2280 180}%
\special{pa 1160 1300}%
\special{fp}%
\special{pa 2220 180}%
\special{pa 1130 1270}%
\special{fp}%
\special{pa 2160 180}%
\special{pa 1100 1240}%
\special{fp}%
\special{pa 2100 180}%
\special{pa 1070 1210}%
\special{fp}%
\special{pa 2040 180}%
\special{pa 1040 1180}%
\special{fp}%
\special{pa 1980 180}%
\special{pa 1020 1140}%
\special{fp}%
\special{pa 1920 180}%
\special{pa 990 1110}%
\special{fp}%
\special{pa 1860 180}%
\special{pa 960 1080}%
\special{fp}%
\special{pa 1800 180}%
\special{pa 930 1050}%
\special{fp}%
\special{pa 1740 180}%
\special{pa 900 1020}%
\special{fp}%
\special{pa 1680 180}%
\special{pa 870 990}%
\special{fp}%
\special{pa 1620 180}%
\special{pa 850 950}%
\special{fp}%
\special{pa 1560 180}%
\special{pa 820 920}%
\special{fp}%
\special{pa 1500 180}%
\special{pa 790 890}%
\special{fp}%
\special{pa 1440 180}%
\special{pa 760 860}%
\special{fp}%
\special{pa 1380 180}%
\special{pa 740 820}%
\special{fp}%
\special{pa 1320 180}%
\special{pa 710 790}%
\special{fp}%
\special{pa 1260 180}%
\special{pa 690 750}%
\special{fp}%
\special{pa 1200 180}%
\special{pa 660 720}%
\special{fp}%
\special{pa 1140 180}%
\special{pa 630 690}%
\special{fp}%
\special{pa 1080 180}%
\special{pa 610 650}%
\special{fp}%
\special{pa 1020 180}%
\special{pa 580 620}%
\special{fp}%
\special{pa 960 180}%
\special{pa 560 580}%
\special{fp}%
\special{pa 900 180}%
\special{pa 530 550}%
\special{fp}%
\special{pa 840 180}%
\special{pa 510 510}%
\special{fp}%
\special{pa 780 180}%
\special{pa 490 470}%
\special{fp}%
\special{pa 720 180}%
\special{pa 460 440}%
\special{fp}%
\special{pa 660 180}%
\special{pa 440 400}%
\special{fp}%
}}%
%
{\color[named]{Black}{%
\special{pn 8}%
\special{pa 600 180}%
\special{pa 410 370}%
\special{fp}%
\special{pa 540 180}%
\special{pa 400 320}%
\special{fp}%
\special{pa 480 180}%
\special{pa 400 260}%
\special{fp}%
}}%
\end{picture}%
    \caption{The shaded area of this figure represents the set $S_1$. 
The set $S_t$ is the set of points $p\in\mathbb{R}^2$ which satisfies $\frac{p}{t}\in S_1$. }
    \label{example2}
  \end{center}
\end{figure}

According to \cite{N}, this $(X, L)$ is an example which admits no Zariski decomposition even after modifications. 
So, it can be expected in this case that the behavior of this multiplier ideal sheaf is different from the algebraic cases.   
Indeed, the set of jumping numbers ${\rm Jump}(\psi; x_0)$ for a point $x$ in $\mathbb{P}(L_0)$ (see \cite[Section 5]{ELSV} for definition) can be written as follows in this case; 
\[
{\rm Jump}(\psi; x_0)=\left\{\left.\frac{p+\sqrt{2p^2a^2-q^2}}{2}\,\right| p, q\in\mathbb{Z},\ 0\leq q<p,\ p-q\equiv 0\ ({\rm mod}\ 2)\right\}, 
\]
which is the set of the largest roots of the quadratic equations
$4T^2-4pT+(1-2a^2)p^2+q^2=0$
of $T$, where integers $p$ and $q$ satisfy the above conditions. 
This set has different properties from algebraic multiplier ideal sheaves. 
For example, it seems difficult to expect the ``periodicity" property, 
and does not have the ``rationality" property in this case 
(For these property, see \cite[1.12]{ELSV} or Remark \ref{propertis_of_jumping_numbers_in_algebraic_cases} below). 
Especially, the singularity exponent $c_{x_0}(\psi)$, which is the minimum number in the set of all jumping numbers, satisfies
\[
c_{x_0}(\psi)=\sqrt{2}a+1, 
\]
and it is clearly irrational. \par
More generally, we give a concrete expression of a minimal singular metric on a big line bundle $L$ on 
the total space of such a toric bundle, see Theorem \ref{main_theorem}. 
As an application, we discuss Zariski closedness of the non-nef locus ${\rm NNef}(L)$ of $L$, see Corollary \ref{NNef_is_Z-closed}. 
\par
The organization of the paper is as follows. 
Let $X$ be the total space of a smooth projective toric bundle over a complex torus, and $L$ be a big line bundle over $X$. 
In Section 2, we recall some facts and notations related to analysis on $X$ and $L$. 
In Section 3, we fix a way to coordinate $X$, and 
study how modifications of $X$ or zeros of holomorphic sections of $L$ can be treated by using this coordinates system. 
In Section 4, we construct a singular hermitian metric $\{e^{-\psi_\sigma}\}$ of $L$ and show it is a minimal singular metric. 
In Section 5, we study some properties related to the positivity of $L$, as applications of the result in Section 4. 
Here we introduce how to calculate the Kiselman numbers and the Lelong numbers of minimal singular metrics, 
and study the non-nef locus of $L$ and multiplier ideal sheaves associated to minimal singular metrics.  
In Section 6, we introduce three examples for $(X, L)$, all of which is based on the example introduced in \cite{N}, 
and apply our result to them. 
\vskip3mm
{\bf Acknowledgment. } 
The author would like to thank his supervisor Prof. Shigeharu Takayama whose enormous support and insightful comments were invaluable during the course of his study. 
He is very grateful to Prof. Shunsuke Takagi for valuable comments and various suggestion. 
He also thanks Tomoyuki Hisamoto and Shin-ichi Matsumura who gave him invaluable comments and warm encouragements. 
He is supported by the Grant-in-Aid for Scientific Research (KAK-
ENHI No. 25-2869) and the Grant-in-Aid for JSPS fellows. 
This work was supported by the Program for Leading Graduate
Schools, MEXT, Japan. 

\section{Preliminaries to analysis on toric bundles}

\subsection{Analysis on compact K\"ahler manifolds}
Let $X$ be a compact K\"ahler manifold and $L$ be a holomorhic line bundle on $X$. 
Let $h$ be a singular hermitian metric on $L$. 
For each local trivialization of $L$ on an open set of $X$, 
``the inner product" defined by $h$ can be written as
$
(\xi, \eta)_z=e^{-\psi(z)}\xi\overline{\eta}
$
where $z$ is a point in the open set, $\xi$ and $\eta$ are points in $\mathbb{C}$, which we regard as the $z$-fiber of $L$, 
and $\psi$ is a locally integrable function defined on the open set, which we call the local weight of $h$.  
The local currents written as $dd^c \psi$ for the local weight $\psi$ of $h$ glue together to define the curvature current associated to $h$. 
We denote it by $\Theta_h$. 

In order to define the minimal singular metric, 
let us recall how to compare the singularities of plurisubharmonic functions. 

\begin{definition}\label{sim_sing} (\cite[1.4]{DPS00})
Let $\varphi$ and $\psi$ be plurisubharmonic functions defined on a neighborhood of $x\in X$. 
We write $\psi\prec_{\rm sing}\varphi$ at $x$ when there exists a positive constant $C$ such that the inequality $e^{-\varphi}\leq Ce^{-\psi}$ holds for each point sufficiently near to $x$. 
We denote $\varphi\sim_{\rm sing}\psi$ at $x$ if $\varphi\prec_{\rm sing}\psi$ and $\varphi\succ_{\rm sing}\psi$ holds at $x$. 
\end{definition}

By using this notation, we can define the minimal singular metric as follows. 

\begin{definition}\label{minimal_singular_metric}
Let $h_{\rm min}$ be a singular hermitian metric on $L$ which satisfies $\Theta_{h_{\rm min}}\geq 0$. 
We call $h_{\rm min}$ a minimal singular metric 
if $\psi\prec_{\rm sing}\varphi_{\rm min}$ holds at any point $x\in X$ 
for all singular hermitian metric $h$ satisfying $\Theta_h\geq 0$, 
where $\varphi_{\rm min}$ and $\psi$ stand for the local weight functions of $h_{\rm min}$ and $h$, respectively, 
with respect to a local trivialization of $L$ around the point $x\in X$. 
\end{definition}

It is known that there exists a minimal singular metric on every pseudo-effective line bundle. 
This fact is proved by considering the upper semi-continuous regularization of the supremum of the all appropriately normalized $\psi$`s, where $\psi$ is as in Definition \ref{minimal_singular_metric} 
(see \cite[1.5]{DPS00} for details). 

Let $L$ be a big line bundle. 
We denote by $N(L)$ the negative part $\sum_{\Gamma\text{\ :\ prime\ divisor}}\nu(\varphi_{\rm min}, \Gamma)\Gamma$ of $L$ in the sense of the divisorial Zariski decomposition \cite{B}, where $\varphi_{\rm min}$ is the local weight of a minimal singular metric on $L$ 
and 
$\nu(\varphi_{\rm min}, \Gamma)$ is the Lelong number of $\varphi_{\rm min}$ at the divisor $\Gamma$. 
We say that $L$ admits a Zariski decomposition if the positive part
$P(L):=c_1(L\otimes\mathcal{O}_X(L))$
is nef class. 
We here remark that this definition of the Zariski-decomposability coincides with Nakayama's algebraic one \cite{N}.

\subsection{Complex tori}
Here, let us recall some fundamental terminologies related to complex tori. 
Let $\Lambda\subset \mathbb{C}^d$ be a lattice. 
We denote $\mathbb{C}^d/\Lambda$ by $V$ and 
the natural map $\mathbb{C}^d\to V$ by $p$. 

\begin{proposition}\label{bl}(\cite[Chapter 3]{BL})
Following four propositions hold for above $d, V$, and $\Lambda$ as above. Here, let us denote by $\mathbb{H}_d$ the set of all hermitian matrices of size $d\times d$ with $\mathbb{C}$-coefficients. \\
$(1)$ There exists an injective $\mathbb{R}$-linear map ${\rm NS}(V)\otimes \mathbb{R}\to \mathbb{H}_d$. \\
$(2)$ By this linear map, ${\rm NS}(V)$ is identified with 
$\{H\in\mathbb{H}_d\mid \forall \lambda , \mu \in\Lambda, {\rm Im}\,(\lambda H\bar{\mu})\in\mathbb{Z}\}$. \\
$(3)$ By this linear map, the nef cone ${\rm Nef}(V)\subset{\rm NS}(V)$ is identified with 
\[
\{H\in\mathbb{H}_d\mid H\geq 0\text{\ and\ }H\text{\ is\ an\ element\ of\ the\ image\ of\ the\ set\ }{\rm NS}(V)\otimes \mathbb{R}\}. 
\]
$(4)$ Let $c_1(E)$ be identified with $H_E\in \mathbb{H}_d$ by this linear map for a line bundle $E$ on $V$. 
Fix a metric $h_E$ of $E$ whose curvature form is a harmonic form with respect to the Euclidean metric 
(such $h_E$ always exists and is unique up to scale). 
Here we fix a point of $V$ and denote by $z=(z_1, z_2, \dots, z_d)$ the local coordinates of $V$ 
around the point induced by the map $p$ and the usual coordinates of $\mathbb{C}^d$. 
Then, there exists a canonically determined local frame $e$ of $E$ on the neighborhood of the point such that, with respect to this local trivialization,  
the local weight function $\varphi_{E}$ of $h_E$ can be written as
\[\varphi_{E}(z_1, z_2, \dots, z_d)=(z_1, z_2, \dots, z_d)H_E\overline{\left(\hskip-2mm\begin{array}{c}
z_1\\
z_2\\
\vdots \\
z_d\\
\end{array}\hskip-2mm\right)}. \]
\end{proposition}

\subsection{Toric bundles}
Here, we review fundamental terminology related to toric bundles. 
We follow \cite[IV]{N} basically. 
Let us denote by $V$ a base complex manifold. 
For simplicity, we restrict ourselves to the case where $V$ is a complex torus. 
Let $N$ be a free $\mathbb{Z}$-module of rank $n$, and $M$ be the dual module ${\rm Hom}(N, \mathbb{Z})$. 
We denote by $e_1, e_2, \dots, e_n$ generators of $N$,and by $e^1, e^2, \dots, e^n$ the dual generators of $M$. 
We write $N_{\mathbb{R}}$ and $M_{\mathbb{R}}$ for $N\otimes\mathbb{R}$ and $M\otimes\mathbb{R}$, respectively. 
We fix a group homomorphism
\[
 \mathcal{L} \colon M\to {\rm Pic}(V)
\]
and a fan $\Sigma$ of $N$, and construct a toric bundle $\pi \colon \mathbb{T}_N(\Sigma, \mathcal{L})\to V$. 
We assume the fan $\Sigma$ is smooth projective, which means that the fan is 
defined by a smooth full-dimensional lattice polytope. 
Under this assumption, the toric variety $\mathbb{T}_N(\Sigma)$ is a smooth projective variety. 
We denote by $\mathcal{L}^m\in{\rm Pic}(V)$ the image of $m\in M$. 
For simplicity, we also denote by $\mathcal{L}^m$ the image of $m\in M_{\mathbb{R}}$ with respect to the linear map 
\[
\mathcal{L}\otimes\mathbb{R} \colon M_\mathbb{R}\to {\rm Pic}(V)\otimes\mathbb{R}. 
\]

\begin{definition}
For $\sigma\in\Sigma$, we define the affine toric bundle $\pi \colon \mathbb{T}_N(\sigma, \mathcal{L})\to V$ by 
\[
 \mathbb{T}_N(\sigma, \mathcal{L})={\rm Spec}_V\hskip-2mm\bigoplus_{m\in \sigma^{\vee}\cap M}\mathcal{L}^m
\]
with the canonical morphism to $V$, and the toric bundle $\pi \colon \mathbb{T}_N(\Sigma, \mathcal{L})\to V$ by gluing $\{\mathbb{T}_N(\sigma, \mathcal{L})\to V\}_{\sigma\in\Sigma}$ in the natural way. 
\end{definition}

For each cone $\sigma\in\Sigma$, 
there exists a corresponding $\mathbb{T}:={\rm Hom}\,(M, \mathbb{C}^*)$-orbit $\mathbb{O}_\sigma(\mathcal{L})$ as the case of toric varieties. 
Let us denote by $\mathbb{V}(\sigma, \mathcal{L})$ the closure of $\mathbb{O}_\sigma(\mathcal{L})$ as the subset of $\mathbb{T}_N(\Sigma, \mathcal{L})$. 
Just as the case of toric varieties, the codimension of $\mathbb{V}(\sigma, \mathcal{L})$ coincides with the dimension of $\sigma$.  
In particular, for each $1$-dimensional $\sigma\in\Sigma$, $\mathbb{V}(\sigma, \mathcal{L})$ is 
a prime divisor of $\mathbb{T}_N(\Sigma, \mathcal{L})$. 

\begin{definition}
We denote by ${\rm Ver}(\Sigma)$ the set of the whole primitive generators $v\in N$ of one-dimensional cones of $\Sigma$. 
For $v\in{\rm Ver}(\Sigma)$, we denote by $\Gamma_v$ the prime divisor $\mathbb{V}(\mathbb{R}_{\geq 0}v, \mathcal{L})$. 
Let us set 
\[
     {\rm PL}_N(\Sigma, \mathbb{Z})=\{h \colon N_\mathbb{R}\to\mathbb{R}\mid \text{for\ each\ }\sigma\in\Sigma,\ h|_\sigma\ \text{is\ linear,\ and\ }h(N)\subset\mathbb{Z}\}. 
\]
For $h\in {\rm PL}_N(\Sigma, \mathbb{Z})$, we define the divisor $D_h$ by
\[
 D_h=\sum_{v\in {\rm Ver}(\Sigma)}(-h(v))\Gamma_v. 
\]
\end{definition}

It is known that any line bundle over $\mathbb{T}_N(\Sigma, \mathcal{L})$ can be written by adding a divisor of the form $D_g$ to the pull-back of a line bundle over $V$ (\cite[2.3]{N}). 

\begin{example}\label{example3_1}
The cone $\{0\}$ is always an element of the fan $\Sigma$. 
Here we consider the affine toric bundle $\mathbb{T}_N(\{0\}, \mathcal{L})$. 
Fix a metric on $\mathcal{L}^{e^j}$ whose curvature form is a harmonic form with respect to the Euclidean metric for each $j$. 
Let $U$ be a sufficiently small open set in $V$ and $z\mapsto s^j(z)$ be such a local trivialization of $\mathcal{L}^{e^j}$ on $U$ as in Proposition \ref{bl}, 
and $z\mapsto s_j(z)$ be the dual frame of the local frame $z\mapsto s^j(z)$ for $j=1, 2, \dots, n$. 
It can be easily checked that the frame $z\mapsto s_j(z)$ is also such a section of $\mathcal{L}^{-e^j}=(\mathcal{L}^{e^j})^{-1}$ as in Proposition \ref{bl}. 
Here, 
\begin{eqnarray}
\mathbb{T}_N(\{0\}, \mathcal{L})|_{\{z\}}&=&{\rm Spec}\,\mathbb{C}[s^1(z), s^2(z), \dots, s^n(z), (s^1)^{-1}(z), 
(s^2)^{-1}(z), \dots, (s^n)^{-1}(z)]\nonumber \\
&=&\prod_{j=1}^n\mathbb{C}^*\cdot s_j(z)\nonumber 
\end{eqnarray}
for $z\in U$. 
Thus, it follows that the affine toric bundle $\mathbb{T}_N(\{0\}, \mathcal{L})$ can be considered as 
the $(\mathbb{C}^*)^n$-bundle on $V$ of which the system $\{s_j\}_j$ works as a local trivialization on $U$. 
\end{example}

\begin{example}\label{example3_2}
Second example is a case where $n=2$.  
Let $L_0, L_1, L_2$ be line bundles over $V$. 
Let $\mathcal{L}$ be a map defined by $e^j\mapsto L_j\otimes L_0^{-1}\,(j=1, 2)$ and  
$\Sigma$ be the fan generated by the three cones
\[
\sigma_1={\rm Cone}\{e_1, e_2\},\ \sigma_2={\rm Cone}\{e_2, -(e_1+e_2)\},\ \text{and}\ \sigma_3={\rm Cone}\{-(e_1+e_2), e_1\}. 
\]
\begin{figure}[ht]
\begin{center}
\unitlength 0.1in
\begin{picture}( 10.1200,  9.0800)(  0.1600,-11.1800)
%
{\color[named]{Black}{%
\special{pn 8}%
\special{pa 590 802}%
\special{pa 590 392}%
\special{fp}%
\special{sh 1}%
\special{pa 590 392}%
\special{pa 570 460}%
\special{pa 590 446}%
\special{pa 610 460}%
\special{pa 590 392}%
\special{fp}%
\special{pa 590 812}%
\special{pa 1000 812}%
\special{fp}%
\special{sh 1}%
\special{pa 1000 812}%
\special{pa 934 792}%
\special{pa 948 812}%
\special{pa 934 832}%
\special{pa 1000 812}%
\special{fp}%
}}%
\put(10.2800,-8.7400){\makebox(0,0)[lb]{$e_1$}}%
\put(4.1000,-3.8000){\makebox(0,0)[lb]{$e_2$}}%
%
{\color[named]{Black}{%
\special{pn 8}%
\special{pa 590 808}%
\special{pa 244 1110}%
\special{fp}%
\special{sh 1}%
\special{pa 244 1110}%
\special{pa 308 1082}%
\special{pa 284 1076}%
\special{pa 282 1052}%
\special{pa 244 1110}%
\special{fp}%
}}%
\put(0.1600,-12.7000){\makebox(0,0)[lb]{$-(e_1+e_2)$}}%
\put(7.5000,-6.4600){\makebox(0,0)[lb]{$\sigma_1$}}%
%
\put(5.7200,-2.1000){\makebox(0,0)[lb]{}}%
\put(1.1200,-7.9600){\makebox(0,0)[lb]{$\sigma_2$}}%
\put(5.9400,-10.6600){\makebox(0,0)[lb]{$\sigma_3$}}%
\end{picture}%
\caption{$\Sigma$. }
\end{center}
\end{figure}
Fix a metric on $\mathcal{L}^{e^j}$ whose curvature form is a harmonic form with respect to the Euclidean metric for each $j$. 
Let $U$ be a sufficiently small open set in $V$ and $z\mapsto s_1(z), z\mapsto s_2(z)$ be such local trivializations of $(L_1\otimes L_0^{-1})^{-1}, (L_2\otimes L_0^{-1})^{-1}$ of $U$  as in Proposition \ref{bl}, respectively, and $s^j$ be the dual of $s_j$ for $j=1, 2$. 
Here, 
\begin{eqnarray}
\mathbb{T}_N(\sigma_1, \mathcal{L})|_{\{z\}}&=&{\rm Spec}\,\mathbb{C}[s^1(z), s^2(z)], \nonumber \\
\mathbb{T}_N(\sigma_2, \mathcal{L})|_{\{z\}}&=&{\rm Spec}\,\mathbb{C}[(s^1(z))^{-1}s^2(z), (s^1(z))^{-1}], \nonumber \\
\mathbb{T}_N(\sigma_3, \mathcal{L})|_{\{z\}}&=&{\rm Spec}\,\mathbb{C}[(s^2(z))^{-1}, s^1(z)(s^2(z))^{-1}], \nonumber
\end{eqnarray}
for $z\in U$. Using this expressions, we can calculate that 
\[
 \mathbb{T}_N(\Sigma, \mathcal{L})=\mathbb{P}(\mathcal{O}_V\oplus(L_1\otimes L_0^{-1})\oplus(L_2\otimes L_0^{-1}))
\cong \mathbb{P}(L_0\oplus L_1\oplus L_2). 
\]
\par
In this case, ${\rm Ver}(\Sigma)$ is the set consisting of the following three elements; 
$v_0=-(e_1+e_2),\ v_1=e_1$, and $v_2=e_2$. 
Let us define $h\in{\rm PL}_N(\Sigma, \mathbb{Z})$ by $v_0\mapsto -1, v_1\mapsto 0$, and $v_2\mapsto 0$. 
Then the line bundle $L=\mathcal{O}_{\mathbb{P}(L_0\oplus L_1\oplus L_2)}(1)$ can be written as 
\[
 L\cong  \pi^*L_0\otimes\mathcal{O}_X(D_h). 
\]
\end{example}

\section{Toric bundles over complex tori}

\subsection{Holomorphic sections and local coordinates}
Let $V$ be a smooth projective variety and $\Sigma$ be the fan defined by a smooth full-dimensional lattice polytope of $M$ just as in the previous section. 
We denote by $X$ the total space of the toric bundle $\pi\colon\mathbb{T}_N(\Sigma, \mathcal{L})\to V$. 
Here we consider holomorphic sections of a line bundle $L$ over $X$. 
According to (\cite[2.3]{N}), without loss of generality, we may assume
$
L= \pi^*L_0\otimes\mathcal{O}_X(D_h), 
$
where $L_0$ is a holomorphic line bundle over $V$, and $h$ is an element of ${\rm PL}_N(\Sigma, \mathbb{Z})$. 


\begin{definition}
We denote by $\Box_h$ the set $\{m\in M_{\mathbb{R}}\mid \forall x\in N_{\mathbb{R}}, \langle m, x\rangle\geq h(x)\}$, 
and by 
$\Box_{\rm Nef}(L_0, h)$ 
the set 
$\{m\in\Box_h\mid L_0\otimes\mathcal{L}^m\text{ \ is\ nef}\}$ 
for a line bundle $L_0$ over $V$ and an element $h\in {\rm PL}_N(\Sigma, \mathbb{Z})$. 
\end{definition}

Since $\Box_h$ is a bounded closed convex set, we clearly obtain the following lemma. 

\begin{lemma}\label{BOX}
$\Box_{\rm Nef}(L_0, h)$ is a bounded closed convex subset of $M_{\mathbb{R}}$. 
\end{lemma}

\begin{definition}
Here we use notations in Example \ref{example3_1}. 
For $m\in M$, we define the meromorphic section $\chi^m$ of $\pi^*\mathcal{L}^{-m}$ on $\mathbb{T}_N(\Sigma, \mathcal{L})$ by 
\[
	(x_j\cdot s_j(z))_j\longmapsto \prod_{j=1}^n (x_j\cdot s_j(z))^{m_j}
=(x_1)^{m_1}\cdot (x_2)^{m_2}\cdot \cdots (x_n)^{m_n}\cdot \left(\prod_{j=1}^n(s^j)^{-m_j}\right)(z)
\]
on $\mathbb{T}_N(\{0\}, \mathcal{L})|_U$, where $m_j=\langle m, e_j\rangle$. 
\end{definition}

$\mathbb{T}_N(\{0\}, \mathcal{L})$, which we considered in Example \ref{example3_1}, 
is always a dense subset of $\mathbb{T}_N(\Sigma, \mathcal{L})$. 
In the case of toric varieties, or the case that $V$ is the ``$0$-dimensional complex torus", 
regular functions on $\mathbb{T}_N(\Sigma, \mathcal{L})$ can be regarded as meromorphic functions on $\mathbb{T}_N(\{0\}, \mathcal{L})$. 
There is an analogue of this fact in the general setting. 

\begin{proposition}(\cite[2.3, 2.4]{N})\label{psdeffrel}
The line bundle $L=\pi^*L_0\otimes\mathcal{O}_X(D_h)$ is pseudo-effective if and only if the set $\Box_{\rm Nef}(L_0, h)$ is non-empty. 
In this case, we obtain the equation
\[
H^0(X, L)=\bigoplus_{m\in\Box_{\rm Nef}(L_0, h)\cap M}\chi^m\cdot\pi^*H^0(V, L_0\otimes\mathcal{L}^m). 
\]
\end{proposition}

In the following, we assume that $V$ is a complex torus. 

\begin{observation}\label{Xmf}
Here we rewrite the meromorphic function $\chi^m\cdot\pi^*f$ in Proposition \ref{psdeffrel} by using notations in Example \ref{example3_1}. 
Let $U$ be a sufficiently small open set in $V$ and $z\mapsto s^0(z)$ be such a local trivialization of $L_0$ on $U$ as in Proposition \ref{bl}. 
Under the local trivialization $z\mapsto \left(s^0\cdot\prod_{j=1}^n s^j\right)(z)$ of $L_0\otimes\mathcal{L}^m$, 
we may assume $f$ is written as
\[
f|_U(z)=\eta(z)\cdot\left(s^0\cdot\prod_{j=1}^n (s^j)^{\langle m, e_j\rangle}\right)(z)
\]
on $U$ for some holomorphic function $\eta$ on $U$. Since 
\[
\chi^m\cdot\pi^*f((x_j\cdot s_j(z))_j)=\chi^m((x_j\cdot s_j(z))_j)\cdot f(z)=\left(\prod_{j=1}^n(x_j)^{\langle m, e_j\rangle}\right)\eta(z)\cdot s^0(z)
\]
holds, it can be checked that $\chi^m\cdot\pi^*f$ is a meromorphic section of $\pi^*L_0$, indeed. 
Moreover we can check that it is an element of
$H^0(X, L)=H^0(X, \pi^*L_0\otimes\mathcal{O}_X(D_h))$, since $m$ is an element of $\Box_h$. 
\end{observation}

In Observation \ref{Xmf}, we calculated $\chi^m\cdot\pi^*f$ as a meromorphic section of $\pi^*L_0$. 
We can rewrite it as a holomorphic section of $\pi^*L_0\otimes\mathcal{O}_X(D_h)$ by using following {\it canonical local coordinates}. 

\begin{definition}\label{canonical_coordinate}
Let $\sigma$ be an element of $\Sigma_{\rm max}:=\{\sigma\in\Sigma\mid{\rm dim}\,\sigma=n\}$. 
Since the fan $\Sigma$ is smooth, there exists $v_1, v_2, \dots, v_n\in{\rm Ver}(\Sigma)$ such that 
$\sigma={\rm Cone}\{v_1, v_2, \dots, v_n\}$
and $v_1, v_2, \dots, v_n$ generates $N$. 
We call such $v_1, v_2, \dots, v_n$ $N$-minimal generators of $\sigma$. 

Let $v^1, v^2, \dots, v^n$ be the dual generators of $v_1, v_2, \dots, v_n$. 
Then the dual cone of $\sigma$ can be written as $\sigma^{\vee}={\rm Cone}\{v^1, v^2, \dots, v^n\}$. 
Fix a metric $h_{v^j}$ of $\mathcal{L}^{v^j}$ whose curvature form is a harmonic form with respect to the Euclidean metric for each $j$. 
Let $U$ be a sufficiently small open set in $V$. 
Let us fix such a local trivializations $z\mapsto t^j(z)$ of $\mathcal{L}^{v^j}$ on $U$ as in Proposition \ref{bl}, 
and the dual section $t_j$ of $t^j$ for $j=1, 2, \dots, n$.  
Using these notations, we can calculate 
\[
\mathbb{T}_N(\sigma, \mathcal{L})|_{\{z\}}=\left.{\rm Spec}\bigoplus_{a_1, a_2, \dots, a_n\geq 0}\mathcal{L}^{\sum_j a_jv^j}\right|_{\{z\}}
={\rm Spec}\,\mathbb{C}[t^1(z), t^2(z), \dots, t^n(z)]
\]
for $z\in U$. 
So, it turns out that $\mathbb{T}_N(\sigma, \mathcal{L})$ is a $\mathbb{C}^n$-bundle which $t_1, t_2, \dots, t_n$ gives a local trivialization on $U$. 
So, we can regard the map 
\[
 (x_1, x_2, \dots, x_n, z)\longmapsto (x_j\cdot t_j(z))_j\in\mathbb{T}_N(\sigma, \mathcal{L})|_{\{z\}}
\]
as a local coordinates system on $\mathbb{T}_N(\sigma, \mathcal{L})|_U$. 
We call this local coordinate system the canonical one of $\mathbb{T}_N(\sigma, \mathcal{L})|_U$ associated to the $N$-minimal generator $v_1, v_2, \dots, v_n$ of $\sigma$. 
\end{definition}

As it is clear from the definition, the canonical coordinates system of $\mathbb{T}_N(\sigma, \mathcal{L})|_U$ associated to the $N$-minimal generator $v_1, v_2, \dots, v_n$ of $\sigma$ depends on the choice of the metrics $\{h_{v^j}\}_j$. 
In the following, we fix basis $e^1, e^2, \dots, e^n$ of $M$ and a metric $h_{e^j}$ of $\mathcal{L}^{e^j}$ whose curvature form is a harmonic form with respect to the Euclidean metric for each $j$, 
and we always choose the metric $h_{e^1}^{\otimes a_1^j}\otimes h_{e^2}^{\otimes a_2^j}\otimes\cdots\otimes h_{e^n}^{\otimes a_n^j}$ for $h_{v^j}$, where $v^j=\sum_ka^j_ke^k$. 
By using this metric, we can say that the canonical coordinates system of $\mathbb{T}_N(\sigma, \mathcal{L})|_U$ associated to the $N$-minimal generator $v_1, v_2, \dots, v_n$ is uniquely determined. 

\begin{remark}
Let $v_1, v_2, \dots, v_n$ be $N$-minimal generators of $\sigma$, 
and $(x_1, x_2, \dots, x_n, z)$ be the canonical coordinates system of $\mathbb{T}_N(\sigma, \mathcal{L})|_U$ associated to $v_1, v_2, \dots, v_n$. 
Then, 
$
 \{x_j=0\}=\Gamma_{v_j} 
$
holds for $j=1, 2, \dots, n$ on $\mathbb{T}_N(\sigma, \mathcal{L})|_U$. 
\end{remark}

\begin{definition}
For $\sigma\in\Sigma_{\rm max}$, we denote by $m_{\sigma}\in M$ the point which satisfies $h(w)=\langle m_{\sigma}, w\rangle$ for all $w\in\sigma$. 
We call $\{m_\sigma\}_{\sigma}$ the Cartier data of $D_h$. 
\end{definition}

\begin{observation}\label{fobs} 
Let $\sigma$ be an element of $\Sigma_{\rm max}$, 
$v_1, v_2, \dots, v_n$ be $N$-minimal generators of $\sigma$, 
and $(x_1, x_2, \dots, x_n, z)$ be the canonical coordinates system of $\mathbb{T}_N(\sigma, \mathcal{L})|_U$ associated to $v_1, v_2, \dots, v_n$. 
In $\mathbb{T}_N(\sigma, \mathcal{L})|_U$, the map 
\[
 (x_1, x_2, \dots, x_n, z)\mapsto \prod_{j=1}^n(x_j)^{\langle m_\sigma, v_j\rangle}
\]
gives a local trivialization of $\mathcal{O}_X(D_h)$, where $\{m_\sigma\}_{\sigma}$ is the Cartier data of $D_h$. 
So, by using notations in Observation \ref{Xmf}, 
\[
  (x_1, x_2, \dots, x_n, z)\longmapsto \left(\prod_{j=1}^n(x_j)^{\langle m_\sigma, v_j\rangle}\right)\cdot s^0(z)
\]
gives a local trivialization of $L$. Under this trivialization, $\chi^m\cdot\pi^*f\in H^0(X, L)$ can be regarded as the holomorphic function 
\[
   (x_1, x_2, \dots, x_n, z)\longmapsto \left(\prod_{j=1}^n(x_j)^{\langle m-m_\sigma, v_j\rangle}\right)\cdot \eta(z)
\]
on $\mathbb{T}_N(\sigma, \mathcal{L})|_U$. 
\end{observation}

The projective line $\mathbb{P}^1=\{[z; w]\}$ can be regarded as the union of two disks $\{[z; 1]\mid |z|\leq 1\}$ and
$\{[1; w]\mid |w|\leq 1\}$ with radius $1$. 
The following proposition is an analogy of this fact. 

\begin{proposition}\label{1cover}
Let $U$ be a sufficiently small open set in $V$, $z_0$ be a point in $U$, 
$\sigma$ be an element of $\Sigma_{\rm max}$, 
$v_1, v_2, \dots, v_n$ be $N$-minimal generators of $\sigma$, 
and $(x_1, x_2, \dots, x_n, z)$ be the canonical coordinates system of $\mathbb{T}_N(\sigma, \mathcal{L})|_U$ associated to $v_1, v_2, \dots, v_n$. 
We set 
\[
K_{\sigma, z_0}=\{(x_1, x_2, \dots, x_n, z_0)\in\mathbb{T}_N(\sigma, \mathcal{L})\mid \forall j\in \{1, 2, \dots, n\}, |x_j|\leq 1\}. 
\]
Then, 
\[
 \bigcup_{\sigma\in\Sigma_{\rm max}}K_{\sigma, z_0} =\pi^{-1}(z_0)
\]
holds. 
\end{proposition}

\begin{proof}
Since $\overline{\mathbb{T}_N(\{0\}, \mathcal{L})|_{\{z_0\}}}=\pi^{-1}(z_0)$, it is sufficient to show that
\[
 \bigcup_{\sigma\in\Sigma_{\rm max}}K_{\sigma, z_0} \supset\mathbb{T}_N(\{0\}, \mathcal{L})|_{\{z_0\}}. 
\]
\par
Let us fix a point $y_0\in\mathbb{T}_N(\{0\}, \mathcal{L})|_{\{z_0\}}$ 
and an element $\tau\in\Sigma_{\rm max}$. 
Let $u_1, u_2, \dots, u_n$ be $N$-minimal generators of $\tau$, 
and $(y_1, y_2, \dots, y_n, z)$ be the canonical coordinates system of $\mathbb{T}_N(\tau, \mathcal{L})|_U$ associated to $u_1, u_2, \dots, u_n$. 
In this coordinates system, assume $y_0$ is written as $((y_0)_1, (y_0)_2, \dots, (y_0)_n, z_0)$. 
Since $y_0\in\mathbb{T}_N(\{0\}, \mathcal{L})$, it turns out that $(y_0)_j\not=0$ for all $j$. 
Thus, 
$
w_0=-\sum_{j=1}^n\log |(y_0)_j|\cdot u_j 
$
defines a point of $N_{\mathbb{R}}$. 
Since $\Sigma$ is complete, there exists an element $\sigma\in\Sigma_{\rm max}$ such that $n_0\in\sigma$. 
Let $v_1, v_2, \dots, v_n$ be $N$-minimal generators of $\sigma$, 
and $(x_1, x_2, \dots, x_n, z)$ be the canonical coordinates system of $\mathbb{T}_N(\sigma, \mathcal{L})|_U$ associated to $v_1, v_2, \dots, v_n$. 
In this coordinates system, $y_0$ can be written as 
\[
 y_0=\left(\left(\prod_{k=1}^n((y_0)_k)^{\langle v^j, u_k\rangle}\right)_j, z_0\right), 
\]
where $v^1, v^2, \dots, v^n$ is the dual basis of $v_1, v_2, \dots, v_n$. 
On the other hands, $w_0$ can be rewritten as 
\[
 w_0=-\sum_{k=1}^n\log |(y_0)_k|\cdot u_k=-\sum_{k=1}^n\sum_{j=1}^n\log |(y_0)_k|\langle v^j, u_k\rangle\cdot v_j
=-\sum_{j=1}^n\log \left|\prod_{k=1}^n((y_0)_k)^{\langle v^j, u_k\rangle}\right|\cdot v_j. 
\]
Since we have chosen $\sigma$ as the condition $n_0\in\sigma$ holds, 
$-\log |\prod_{k=1}^n((y_0)_k)^{\langle v^j, u_k\rangle}|\geq 0$, or $|\prod_{k=1}^n((y_0)_k)^{\langle v^j, u_k\rangle}|\leq 1$
holds for all $j\in\{1, 2, \dots, n\}$. 
We thus obtain $y_0\in K_{\sigma, z_0}$, which proves the proposition. 
\end{proof}

\subsection{Modifications}
Let $\Sigma$ be a smooth projective fan of the $n$-dimensional lattice $N$. 
Here we fix a smooth subdivision fan $\tilde{\Sigma}$ of $\Sigma$, 
and consider a toric bundle $\tilde{X}=\mathbb{T}_N(\tilde{\Sigma}, \mathcal{L})$ and the canonical morphism $\mu\colon \tilde{X}\rightarrow X$.  
As in the case of toric varieties, $\mu\colon \tilde{X}\rightarrow X$ is a proper modification of $X$. 
From this section, we use letters with subscripts such as $v_1, v_2, \dots, v_n$ for generators of $N$, 
and we denote the dual generators by the same letters with superscripts, such as $v^1, v^2, \dots, v^n$, throughout this paper. 

First of all, we obtain the following result by simple computations. 

\begin{lemma}\label{mu}
Let $\sigma\in\Sigma_{\rm max}, \tilde{\sigma}\in\tilde{\Sigma}_{\rm max}$ be cones such that $\tilde{\sigma}\subset\sigma$, 
$v_1, v_2, \dots, v_n$ be $N$-minimal generators of $\sigma$, 
and $\tilde{v}_1, \tilde{v}_2, \dots, \tilde{v}_n$ be $N$-minimal generators of $\tilde{\sigma}$. 
We denote by $(x_1, x_2, \dots, x_n, z)$ and $(\tilde{x}_1, \tilde{x}_2, \dots, \tilde{x}_n, z)$ 
the canonical coordinates systems of $\mathbb{T}_N(\sigma, \mathcal{L})|_U$ and $\mathbb{T}_N(\tilde{\sigma}, \mathcal{L})|_U$, respectively. 
In these coordinates, the morphism $\mu\colon \tilde{X}\rightarrow X$ can be written as
\[
\mu(\tilde{x}_1, \tilde{x}_2, \dots, \tilde{x}_n, z)=\left(\left(\prod_{k=1}^n(\tilde{x}_k)^{\langle v^j, \tilde{v}_k\rangle}\right)_j, \ z\right). 
\]
\end{lemma}

Lemma \ref{mu} immediately implies the following corollary. 

\begin{corollary}\label{jvj}
For $j\in\{1, 2, \dots, n\}$, there exists a subset $J_{v_j}\subset\{1, 2, \dots, n\}$ such that 
$
\mu^*\Gamma_{v_j}=\bigcup_{k\in J_{v_j}}\{\tilde{x}_k=0\}
$
in $\mathbb{T}_N(\tilde{\sigma}, \mathcal{L})|_U$. 
\end{corollary}

\begin{remark}\label{hyouji_jvj}
For Corollary \ref{jvj}, the set $J_{v_j}$ can be written as
\[
 J_{v_j}=\{k\in\{1, 2, \dots, n\}\mid\langle v^j, \tilde{v}_k\rangle\not=0\}. 
\]
\end{remark}

For $\sigma\in\Sigma_{\rm max}$, we define the set $\tilde{\Sigma}_\sigma$ by
$\tilde{\Sigma}_\sigma:=\{\tilde{\sigma}\in\tilde{\Sigma}\mid \tilde{\sigma}\subset\sigma\}, $
and we denote by $(\tilde{\Sigma}_\sigma)_{\rm max}$ the set $\{\tilde{\sigma}\in\tilde{\Sigma}_\sigma\mid {\rm dim}\,\tilde{\sigma}=n\}$. 
By using the expression of $\mu$ in Lemma \ref{mu}, we can get the following lemma. 

\begin{lemma}\label{coord}
Fix a point $z_0\in U$, a set $I\subset\{1, 2, \dots, n\}$, and a cone $\sigma\in\Sigma_{\rm max}$. 
Denote by $W_{I, \sigma, z_0}$ the set 
\[
\{(x_1, x_2, \dots, x_n, z_0)\in\mathbb{T}_N(\sigma, \mathcal{L})\mid \forall j\in I, |x_j|\leq 1,\  \forall j\in \{1, 2, \dots, n\}, x_j\not=0\}, 
\]
and by $W_{I, \tilde{\sigma}, z_0}$ the set 
\[
\{(\tilde{x}_1, \tilde{x}_2, \dots, \tilde{x}_n, z_0)\in\mathbb{T}_N(\tilde{\sigma}, \mathcal{L})
\mid \forall k\in \cup_{j\in I}J_{v_j}, |\tilde{x}_k|\leq 1, \  \forall j\in \{1, 2, \dots, n\}, \tilde{x}_j\not=0\}
\]
for each $\tilde{\sigma}\in(\tilde{\Sigma}_\sigma)_{\rm max}$. 
Then, 
\[
\mu\left(\bigcup_{\tilde{\sigma}\in(\tilde{\Sigma}_\sigma)_{\rm max}}W_{I, \tilde{\sigma}, z_0}\right)=W_{I, \sigma, z_0}
\]
holds.  
\end{lemma}

This lemma can be proved in the almost same way as those used in Lemma \ref{1cover}. 
Applying this lemma with $I=\{1, 2, \dots, n\}$, we obtain the next corollary. 

\begin{corollary}\label{coord1}
Here we use notations in Lemma \ref{coord}. 
Denote by $K_\sigma$ the set 
\[
\{(x_1, x_2, \dots, x_n, z)\in\mathbb{T}_N(\sigma, \mathcal{L})|_{\overline{U}}\mid \forall j\in \{1, 2, \dots, n\}, |x_j|\leq 1\}
\]
and by $K_{\tilde{\sigma}}$ the set 
\[
\{(\tilde{x}_1, \tilde{x}_2, \dots, \tilde{x}_n, z)\in\mathbb{T}_N(\tilde{\sigma}, \mathcal{L})|_{\overline{U}}\mid \forall j\in\{1, 2, \dots, n\}, |\tilde{x}_j|\leq 1\}
\]
for each $n$-dimensional cone $\tilde{\sigma}\in\tilde{\Sigma}_\sigma$. 
Then, 
\[
\mu\left(\bigcup_{\tilde{\sigma}\in(\tilde{\Sigma}_\sigma)_{\rm max}}K_{\tilde{\sigma}}\right)=K_\sigma
\]
holds. 
\end{corollary}

\subsection{Convex subsets of $M$}
Let $\Sigma$ be a smooth projective fan of the $n$-dimensional lattice $N$,  $\sigma\in \Sigma$ be a $n$-dimensional cone, 
$v_1, v_2, \dots, v_n$ be $N$-minimal generators of $\sigma$, 
and $(x_1, x_2, \dots, x_n, z)$ be the canonical coordinates system of $\mathbb{T}_N(\sigma, \mathcal{L})|_U$ associated to $v_1, v_2, \dots, v_n$, 
where $U$ is a sufficiently small open set in $V$. 
\begin{definition}
For $A\subset\sigma^{\vee}$, we denote by $\overline{\overline{A}}$ the set 
\[
\{m\in\sigma^{\vee}\mid \forall w\in\sigma, \min_{m'\in A}\langle m', w\rangle\leq\langle m, w\rangle\}. 
\]
When $A=\emptyset$, we formally regards $\overline{\overline{\emptyset}}$ as $\sigma^\vee$. 
\end{definition} 

\begin{definition}
Let $m_\sigma$ be an element of the Cartier data $D_h$ which is associated to $\sigma$. 
We denote by $S(L_0, h)_\sigma$ the subset 
$
\overline{\overline{\{m-m_\sigma\mid m\in\Box_{\rm Nef}(L_0, h)\}}}\subset\sigma^{\vee}. 
$
\end{definition} 

\begin{remark}\label{Srmk}
In $\prod_{j\in I}\{|x_j|< 1\}\times\prod_{j\notin I}\{x_j\in \mathbb{C}\}\times U$, 
\[
\max_{m\in S(L_0, h)_\sigma}\prod_{j\in I}|x_j|^{2\langle m, v_j\rangle}
=\max_{m\in\Box_{\rm Nef}(L_0, h)}\prod_{j\in I}|x_j|^{2\langle m-m_\sigma, v_j\rangle}
\]
for any $I\subset\{1, 2, \dots, n\}$, 
where $m_\sigma$ is an element of the Cartier data $D_h$ which is associated to $\sigma$. 
\qed
\end{remark}

\begin{definition}\label{Pdef}
For a point $((x_0)_1, (x_0)_2, \dots, (x_0)_n, z_0)\in\mathbb{T}_N(\sigma, \mathcal{L})|_U$, 
let us denote by $I$ the set $\{j\in\{1, 2, \dots, n\}\mid x_0^j=0\}$. 
We define the set $P(f_1, f_2, \dots, f_l)_{((x_0)_1, (x_0)_2, \dots, (x_0)_n, z_0)}$ for 
$f_1, f_2, \dots, f_l\in\mathcal{O}_{((x_0)_1, (x_0)_2, \dots, (x_0)_n, z_0)}$ as follows. 
Let 
\[
f_\nu(x_1, x_2, \dots, x_n)=\sum_{\alpha \geq 0}(x_I)^\alpha A_{\nu, \alpha} (x_{I^c}, z), 
\]
be the Taylor expansion of each $f_\nu\,(\nu=1, 2, \dots, l)$ around the point $((x_0)_1, (x_0)_2, \dots, (x_0)_n, z_0)$ for variables $\{x_j\}_{j\in I}$,
where $\alpha=(a_j)_{j\in I}$ is a multi-index, the signature ``$(x_I)^\alpha$" stands for $\prod_{j\in I}(x_j)^{a_j}$, 
and $A_{\nu, \alpha}$ is the germ of a  holomorphic function with $(n-\#I+d)$-variables $(x_{I^c}, z)=((x_j)_{j\not\in I}, z)$. 
We define $P(f_1, f_2, \dots, f_l)_{((x_0)_1, (x_0)_2, \dots, (x_0)_n, z_0)}$ by 
\[
P(f_1, f_2, \dots, f_l)_{((x_0)_1, (x_0)_2, \dots, (x_0)_n, z_0)}=\overline{\overline{\left. \bigcup_{\nu=1}^l\left\{\sum_{j\in I}a_j\cdot v^j\right| A_{\nu, (a_j)_j}\not\equiv 0\right\}}}
\subset\sigma^\vee. 
\]
\end{definition} 

\begin{remark}\label{Pinv}
Here, we use notations in Definition \ref{Pdef}. 
Set 
\[
 P_\sigma=P(f_1, f_2, \dots, f_l)_{(0, 0, \dots, 0, z_0)}
\]
for $(0, 0, \dots, 0, z_0)\in\mathbb{T}_N(\sigma, \mathcal{L})|_U$. 
Let $\tilde{\Sigma}$ be a smooth complete fan which is a subdivision of $\Sigma$, 
$\tilde{\sigma}\in\tilde{\Sigma}_{\rm max}$ be a cone such that $\tilde{\sigma}\subset\sigma$, 
$\tilde{v}_1, \tilde{v}_2, \dots, \tilde{v}_n$ be $N$-minimal generators of $\tilde{\sigma}$, 
and $(\tilde{x}_1, \tilde{x}_2, \dots, \tilde{x}_n, z)$ be the canonical coordinates system of $\mathbb{T}_N(\tilde{\sigma}, \mathcal{L})|_U$ associated to $\tilde{v}_1, \tilde{v}_2, \dots, \tilde{v}_n$. 
For the point $(0, 0, \dots, 0, z_0)$, let us set 
\[
 P_{\tilde{\sigma}}=P(\mu^*f_1, \mu^*f_2, \dots, \mu^*f_l)_{(0, 0, \dots, 0, z_0)}, 
\]
and assume that $f_\nu$ is expanded as
\[
f_\nu(\tilde{x}_1, \tilde{x}_2, \dots, \tilde{x}_n, z)=\sum_{(a_j)_j \geq 0}\prod_{j=1}^n(x_j)^{a_j} A_{\nu, (a_j)_j} (z)
\]
around $(0, 0, \dots, 0, z_0)$. 
Then, by Lemma \ref{mu}, $\mu^*f_\nu$ can be written as 
\[
\mu^*f_\nu(\tilde{x}_1, \tilde{x}_2, \dots, \tilde{x}_n, z)=\sum_{(a_j)_j \geq 0}\prod_{k=1}^n(\tilde{x}_k)^{\sum_{j=1}^na_j\langle v^j, \tilde{v}_k\rangle} A_{\nu, (a_j)_j} (z)
\]
around $(0, 0, \dots, 0, z_0)$. 
Thus, it follows that the following two sets are same; 
\[
\bigcup_{\nu=1}^l\left\{\left. \sum_{j=1}^na_j\cdot v^j\right| A_{\nu, (a_j)_j}\not\equiv 0\right\}
=\bigcup_{\nu=1}^l\left\{\left. \sum_{j, k=1}^na_j\langle v^j, \tilde{v}_k\rangle\cdot \tilde{v}^k\right| A_{\nu, (a_j)_j}\not\equiv 0\right\}. 
\]
However, since the two signature $\overline{\overline{\cdot}}$ appeared in the definition of $P_\sigma$ and $P_{\tilde{\sigma}}$ are different from each other,  
we can not say nothing more than $P_\sigma\subset P_{\tilde{\sigma}}$ in general. 
\end{remark}

\begin{remark}\label{Prmk}
Here, we use notations in Definition \ref{Pdef}. 
We remark that \linebreak $P(f_1, f_2, \dots, f_l)_{((x_0)_1, (x_0)_2, \dots, (x_0)_n, z_0)}$ is finitely generated in the following sense; 
There exists a finite subset 
\[
\{m_1, m_2, \dots, m_l\}\subset P(f_1, f_2, \dots, f_l)_{((x_0)_1, (x_0)_2, \dots, (x_0)_n, z_0)}\cap\bigoplus_{j=1}^n\mathbb{Z}_{\geq 0}v^j
\]
of the lattice such that the equation
\[
P(f_1, f_2, \dots, f_l)_{((x_0)_1, (x_0)_2, \dots, (x_0)_n, z_0)}=\overline{\overline{\{m_1, m_2, \dots, m_l\}}}
\]
holds. 
More generally, for any subset 
$
A\subset\bigoplus_{j=1}^n\mathbb{Z}_{\geq 0} v^j, 
$
there exists a finite subset
\[
\{m_1, m_2, \dots, m_l\}\subset \overline{\overline{A}}\cap\bigoplus_{j=1}^n\mathbb{Z}_{\geq 0} v^j
\]
of lattice points such that the equation
$
\overline{\overline{A}}=\overline{\overline{\{m_1, m_2, \dots, m_l\}}}
$
holds. 
\end{remark}

\begin{lemma}\label{lem2}
For each finite set $A\subset\bigoplus_{j=1}^n\mathbb{Q}_{\geq 0} v^j$ of rational points, there exists a smooth complete cone $\tilde{\Sigma}$ 
which satisfies the following two conditions (i) and (ii). 
(i) $\tilde{\Sigma}$ is a subdivision of $\Sigma$. 
(ii) For all $n$-dimensional cone $\tilde{\sigma}\in\tilde{\Sigma}$ satisfying $\tilde{\sigma}\subset\sigma$, 
there exists an element $m_0\in A$ such that 
$
\min_{m\in\overline{\overline{A}}}\langle m, w\rangle
=\langle m_0, w\rangle 
$
holds for all $w\in\tilde{\sigma}$, 
where $\tilde{v}_1, \tilde{v}_2, \dots, \tilde{v}_n$ is $N$-minimal generators of $\tilde{\sigma}$. 
\end{lemma} 
\begin{proof}
Let $\tilde{\Sigma}$ be a fan which is made by cutting all cones of $\Sigma$ by the all hyperplanes
\[
\{w\in N_\mathbb{R}\mid \langle m_j, w\rangle=\langle m_k, w\rangle\}\ (m_j, m_k\in A)
\]
of $N_\mathbb{R}$. 
Since $A\subset\bigoplus_{j=1}^n\mathbb{Q}_{\geq 0}v^j$, each cone of $\tilde{\Sigma}$ is rational. 
Moreover, for all $n$-dimensional cone of $\tilde{\Sigma}$ satisfying $\tilde{\sigma}\subset\sigma$, 
there exists an element $m_{\tilde{\sigma}}\in A$ such that
$\min_{m\in\overline{\overline{A}}}\langle m, w\rangle
=\langle m_{\tilde{\sigma}}, w\rangle$ holds for all $w\in\tilde{\sigma}$. 
Let $\tilde{\Sigma}'$ be a smooth fan which is a subdivision of $\tilde{\Sigma}$. 
This fan $\tilde{\Sigma}'$ is what we desired.  
\end{proof} 

\section{Construction of minimal singular metrics}
Here, we use notations in the previous section. 
In this section, we construct a minimal singular metric on the big line bundle $L=\pi^*L_0\otimes\mathcal{O}_X(D_h)$ over the total space of a toric bundle $X=\mathbb{T}_N(\Sigma, \mathcal{L})$
over a complex torus $V$, where $\Sigma$ is a smooth projective fan in a $n$-dimensional fan $N$. 
According to Proposition \ref{psdeffrel}, it is clear that the set $\Box_{\rm Nef}(L_0, h)=\Box_{\rm Nef}(L_0, h)$ is not empty in this setting. \par
First of all, we define the singular hermitian metric $e^{-\psi_{\sigma, m}}$ for each $m\in\Box_{\rm Nef}(L_0, h)$. 

\begin{definition}\label{psisigmamdef}
Let $m$ be an element of $\Box_{\rm Nef}(L_0, h)$, 
$\sigma$ be an element of $\Sigma_{\rm max}$, 
$v_1$, $v_2$, $\dots$, $v_n$ be $N$-minimal generators of $\sigma$, 
and $\{m_\sigma\}_{\sigma}$ be the Cartier data of $D_h$. 
Here, we define the plurisubharmonic function $\psi_{\sigma, m}$ on $\mathbb{T}_N(\sigma, \mathcal{L})|_U$ by 
\[
 \psi_{\sigma, m}(x_1, x_2, \dots, x_n, z)= \log\left(\prod_{j=1}^n|x_j|^{2\langle m-m_\sigma, v_j\rangle}\right)+\varphi_{L_0\otimes\mathcal{L}^m}(z), 
\]
where $U$ is a sufficiently small open set in $V$ and $\varphi_{L_0\otimes\mathcal{L}^m}=\varphi_{L_0}+\sum_{j=1}^n\langle m, v_j\rangle\varphi_{\mathcal{L}^{v_j}}$.  
For the definition of $\varphi_{L_0}$ and $\varphi_{\mathcal{L}^{v_j}}$, see Proposition \ref{bl}.  
And here, we formally regard $0^0$ as $1$. 
\end{definition}

\begin{remark}
In Definition \ref{psisigmamdef}, the first term of the defining equation of $\psi_{\sigma, m}$ is clearly plurisubharmonic. 
According to Proposition \ref{bl}, the second term is also turned out to be plurisubharmonic. 
Thus $\psi_{\sigma, m}$ is also a plurisubharmonic function, indeed. 
\end{remark}

\begin{remark}\label{psisigmam}
The functions $\{e^{-\psi_{\sigma, m}}\}_{\sigma\in\Sigma_{\rm max}}$ glue together to give a singular hermitian metric on $L$. 
Here, we explain this fact when $m$ is a rational point of $M_{\mathbb{R}}$ for simplicity. \par
Let $\nu$ be a natural number such that $\nu m\in M$. By Observation \ref{Xmf}, $\nu\psi_{\sigma, m}$ can be rewritten as
\[
 \nu\psi_{\sigma, m}=\log |\chi^{\nu m}|^2+\nu\varphi_{L_0\otimes\mathcal{L}^m}. 
\]
Since $\chi^{\nu m}$ can be regarded as a meromorphic section of the line bundle 
$\mathcal{O}_X(D_{\nu h})\otimes\pi^*\mathcal{L}^{-\nu m}$, 
the first term of the right hand side of the above equation is turned out to be a local weight of a singular hermitian metric which is defined globally on $\mathcal{O}_X(D_{\nu h})\otimes\pi^*\mathcal{L}^{-\nu m}$. 
Since the second term is also a local weight of the hermitian metric globally defined on $\pi^*(L_0^{\nu}\otimes\mathcal{L}^{\nu m})$, 
the sum $\nu\psi_{\sigma, m}$ is a local weight of a singular hermitian metric globally defined on $\nu L=\pi^*L_0^{\nu}\otimes\mathcal{O}_X(D_{\nu h})$. \par
This explanation also makes sense in the general case, by considering formally with $\mathbb{R}$-line bundles. 
\end{remark}

\begin{definition} 
We define the plurisubharmonic function $\psi_\sigma$ on $\mathbb{T}_N(\sigma, \mathcal{L})|_U$ by 
\[
\psi_\sigma(x_1, x_2, \dots, x_n, z)=\max_{m\in\Box_{\rm Nef}(L_0, h)}\psi_{\sigma, m}(x_1, x_2, \dots, x_n, z)
\]
for a sufficiently small open set $U$ of $V$ and $\sigma\in\Sigma_{\rm max}$. 
\end{definition}

\begin{remark}
Since each $\psi_{\sigma, m}$ is plurisubharmonic, it is clear that the upper envelope 
\[
(x_1, x_2, \dots, x_n, z)\longmapsto\limsup_{(\xi^1, \xi^2, \dots, \xi^n, \zeta)\to(x_1, x_2, \dots, x_n, z)}\psi_\sigma(\xi^1, \xi^2, \dots, \xi^n, \zeta)
\]
of $\psi_\sigma$ is a plurisubharmonic function. 
Now let us consider the function 
\[
((x_1, x_2, \dots, x_n, z), m)\longmapsto e^{\psi_{\sigma, m}(x_1, x_2, \dots, x_n, z)}= \left(\prod_{j=1}^n|x_j|^{2\langle m-m_{\sigma}, v_j\rangle}\right)\cdot e^{\varphi_{L_0\otimes\mathcal{L}^m}(z)}. 
\]
This function is a continuous function defined on $\mathbb{T}_N(\sigma, \mathcal{L})|_U\times \Box_{\rm Nef}(L_0, h)$. 
Since $\Box_{\rm Nef}(L_0, h)$ is compact (Lemma \ref{BOX}), the function 
\[
((x_1, x_2, \dots, x_n, z), m)\longmapsto e^{\psi_\sigma(x_1, x_2, \dots, x_n, z)}=\max_{m\in\Box_{\rm Nef}(L_0, h)} e^{\psi_{\sigma, m}(x_1, x_2, \dots, x_n, z)}, 
\]
is also continuous. 
Therefore, $\psi_\sigma$ itself is also a plurisubharmonic function. 
\end{remark}

\begin{remark}
Remark \ref{psisigmam} yields that $\{e^{-\psi_\sigma}\}_{\sigma\in \Sigma_{\rm max}}$ glue together to give a singular hermitian metric on $L$ whose curvature current is semi-positive.   
\end{remark}

\begin{theorem}\label{main_theorem}
Assume that $L$ is a big line bundle, then the singular hermitian metric $e^{-\psi_\sigma}$ of $L$ is a minimal singular metric. 
\end{theorem}

From now on, we will prepare for the proof of Theorem \ref{main_theorem}. 
Let $\sigma\in\Sigma$ be a $n$-dimensional cone, 
$v_1, v_2, \dots, v_n$ be $N$-minimal generators of $\sigma$, 
and $(x_1, x_2, \dots, x_n, z)$ be the canonical coordinates system of $\mathbb{T}_N(\sigma, \mathcal{L})|_U$ associated to $v_1, v_2, \dots, v_n$, 
where $U$ is a sufficiently small open set in $V$. 
We use these notations throughout this section.  

\begin{lemma}\label{lem1}
Let us fix a point $((x_0)_1, (x_0)_2, \dots, (x_0)_n, z_0)\in\mathbb{T}_N(\sigma, \mathcal{L})|_U$, 
and denote by $I$ the set $\{j\in\{1, 2, \dots, n\}\mid x_0^j=0\}$. 
Then, there exist constants $C_1$ and $C_2$ such that 
\[
\max_{m\in\Box_{\rm Nef}(L_0, h)}\log\prod_{j\in I}|x_j|^{2\langle m-m_\sigma, v_j\rangle}+C_1\leq \psi_\sigma
\leq\max_{m\in\Box_{\rm Nef}(L_0, h)}\log\prod_{j\in I}|x_j|^{2\langle m-m_\sigma, v_j\rangle}+C_2
\]
holds on $\prod_{j\in I}^n\{|x_j|\leq 1\}\times\prod_{j\notin I}\{|x_j-x_0^j|\leq \delta _j\}\times \overline{U}$, 
where $\{\delta_j\}_{j\not\in I}$ is a system of sufficiently small positive numbers such that $0\not\in\{|x_j-x_0^j|\leq \delta _j\}$ for all $j\not\in I$, 
and $m_\sigma$ is the element of the Cartier data of $D_h$ which is associated to $\sigma$. 
\end{lemma} 

\begin{proof}
The function 
\[
\left(m, (x_j)_{j\not\in I}, z\right)\longmapsto
\log\prod_{j\not\in I}|x_j|^{2\langle m-m_\sigma, v_j\rangle}+\varphi_{L_0\otimes\mathcal{L}^m}(z)
\]
defined on $\Box_{\rm Nef}(L_0, h)\times\prod_{j\notin I}\{|x_j-x_0^j|\leq \delta _j\}\times \overline{U}$ is continuous. 
According to Lemma \ref{BOX}, $\Box_{\rm Nef}(L_0, h)\times\prod_{j\notin I}\{|x_j-x_0^j|\leq \delta _j\}\times \overline{U}$ is compact, which yields that this function has both the maximum value and the minimum value, which we denote by $C_1$ and $C_2$ respectively. 
Therefore, the inequality 
\[
\log\prod_{j\in I}|x_j|^{2\langle m-m_\sigma, v_j\rangle}+C_1\leq
\psi_{\sigma, m}\leq
\log\prod_{j\in I}|x_j|^{2\langle m-m_\sigma, v_j\rangle}+C_2
\]
follows, which proves the lemma. 
\end{proof}

As we have assumed that $L$ is big thus in particular pseudo-effective, there must be a minimal singular metric on $L$. 
We fix one of these and denote it by $h_{\rm min}$. 

\begin{lemma}\label{mainlem}
Let $\sigma$ be an element of $\Sigma_{\rm max}$, 
and we denote the weight function of $h_{\rm min}$ around $\mathbb{T}_N(\sigma, \mathcal{L})|_{\overline{U}}$ 
with respect to the local trivialization of $L$ as in Observation \ref{Xmf} by $\varphi_{\min, \sigma}$. 
Then, there exists a constant $C_\sigma$ such that 
\[
 \varphi_{\rm min, \sigma}\leq \psi_\sigma+C_\sigma
\]
holds on the set $K_\sigma=\{(x_1, x_2, \dots, x_n, z)\in\mathbb{T}_N(\sigma, \mathcal{L})|_{\overline{U}}\mid \forall j\in \{1, 2, \dots, n\}, |x_j|\leq 1\}$. 
\end{lemma}
\begin{proof}
Let us denote by $m_\sigma$ the element of the Cartier data of $D_h$ associated to $\sigma$. 
Applying Lemma \ref{lem1} with $I=\{1, 2, \dots, n\}$, it follows that there exists a constant $C$ such that 
\[
\max_{m\in\Box_{\rm Nef}(L_0, h)}\log\prod_{j=1}^n|x_j|^{2\langle m-m_\sigma, v_j\rangle}\leq \psi_\sigma +C
\]
holds on $K_\sigma$. \par
Thus here, we compare $\varphi_{\rm min, \sigma}$ with $\max_{m\in\Box_{\rm Nef}(L_0, h)}\log\prod_{j=1}^n|x_j|^{2\langle m-m_\sigma, v_j\rangle}$. \par 
We choose an infinite subsequence $\{\nu\}\subset\mathbb{N}$ and a finite subset $\{f_j^{(\nu)}\}_{1\leq j\leq N_\nu}$ of $H^0(X, \nu L)$ for each $\nu$ satisfying the following condition; 
The function  
\[
\varphi _\nu=\frac{1}{\nu}\log\sum_{j=1}^{N_\nu}|f_j^{(\nu)}|^2
\]
converges pointwise to $\varphi_{\rm min, \sigma}$ on $X$ except a subset of measure $0$ as $\nu\to \infty$, 
and the maximum value $M_{\varphi_\nu}$ of $\varphi_\nu$ on $K_\sigma$ also converges 
to $M_{\varphi_{\rm min}, \sigma}=\max_{K_\sigma}\varphi_{\rm min, \sigma}$ as $\nu\to \infty$. 
The existence of these functions can be immediately shown by applying \cite[Theorem (13.21)]{D} regarding $\varphi$ in the theorem 
as $(1-\frac{1}{k})\varphi_{\rm min}+\frac{1}{k}\varphi_+$ for each natural number $k$, where $\varphi_+$ is the local weight of a singular hermitian metric $h_+$ on $L$ 
which satisfies $\Theta_{h_+}\geq\varepsilon\omega$ for some positive number $\varepsilon$ and a K\"ahler metric $\omega$ on $X$. 

Then, according to the next Lemma \ref{lem3}, an inequality
\[
\varphi_\nu\leq \max_{m\in\Box_{\rm Nef}(L_0, h)}\log\prod_{j=1}^n|x_j|^{2\langle m-m_\sigma, v_j\rangle}+M_{\varphi_\nu}
\]
holds on $K_\sigma$. 
Considering this inequality as $\nu\to\infty$, we obtain 
\[
\varphi_{\rm min, \sigma}\leq \max_{m\in\Box_{\rm Nef}(L_0, h)}\log\prod_{j=1}^n|x_j|^{2\langle m-m_\sigma, v_j\rangle}+M_{\varphi_{\rm min}, \sigma}
\]
on $K_\sigma$ except the subset of measure $0$. 
Since the both hand sides are plurisubharmonic, this inequality holds on whole $K_\sigma$. \par
According to the above argument, we obtain the inequality
\[
\varphi_{\rm min, \sigma}\leq \psi_\sigma +C+M_{\varphi_{\rm min}, \sigma}
\]
on $K_\sigma$, which proves the lemma. 
\end{proof}

\begin{lemma}\label{lem3}
Here we use notations in the proof of Lemma \ref{mainlem}. 
The inequality
\[
\varphi_\nu\leq \max_{m\in\Box_{\rm Nef}(L_0, h)}\log\prod_{j=1}^n|x_j|^{2\langle m-m_\sigma, v_j\rangle}+M_{\varphi_\nu}
\]
holds on $K_\sigma$. 
\end{lemma} 

\begin{proof}
Let $P(\varphi_\nu)_\sigma:=\frac{1}{\nu}P(f^{(\nu)}_1, f^{(\nu)}_2, \dots, f^{(\nu)}_{N_\nu})_{(0, 0, \dots, 0, z_0)}$. 
According to Proposition \ref{psdeffrel} and Observation \ref{fobs}, $\nu P(\varphi_\nu)_\sigma$ is a subset of $S(L_0^{\nu}, \nu h)_\sigma$. 
Since $\Box_{\rm Nef}(L_0^{\nu}, \nu h)=\nu \Box_{\rm Nef}(L_0, h)$ holds, it turns out that $S(L_0^{\nu}, \nu h)_\sigma=\nu S(L_0, h)_\sigma$
, thus we obtain
\[
P(\varphi_\nu)_\sigma\subset S(L_0, h)_\sigma . 
\]
Therefore, according to Remark \ref{Srmk}, it is sufficient to show the inequality 
\[
\varphi_\nu\leq \max_{m\in P(\varphi_\nu)_\sigma}\log\prod_{j=1}^n|x_j|^{2\langle m, v_j\rangle}+M_{\varphi_\nu}
\]
on $K_\sigma$. 

According to Remark \ref{Prmk}, there exists a finite subset $A$ of $P(\varphi_\nu)$ whose elements are rational and which satisfies $P(\varphi_\nu)=\overline{\overline{A}}$. 
For this set $A$, we fix such a subdivision $\tilde{\Sigma}$ of $\Sigma$ as in Lemma \ref{lem2}. 
In the following, we use notations we used in Section 4.2.
According to Corollary \ref{coord1}, it is sufficient to show that 
\[
\mu^*\varphi_\nu\leq \mu^*\left(\max_{m\in P(\varphi_\nu)_\sigma}\log\prod_{j=1}^n|x_j|^{2\langle m, v_j\rangle}\right)+M_{\varphi_\nu}
\]
on $K_{\tilde{\sigma}}=\{(\tilde{x}_1, \tilde{x}_2, \dots, \tilde{x}_n, z)\in\mathbb{T}_N(\tilde{\sigma}, \mathcal{L})|_{\overline{U}}\mid \forall j\in\{1, 2, \dots, n\}, |\tilde{x}_j|\leq 1\}$
for each $\tilde{\sigma}\in(\tilde{\Sigma}_\sigma)_{\rm max}$. \par
Since 
\[
\log\prod_{j=1}^n|\mu^*x_j|^{2\langle m, v_j\rangle}
=\log\prod_{j=1}^n\prod_{k=1}^n|\tilde{x}_k|^{2\langle m, v_j\rangle\langle v^j, \tilde{v}_k\rangle}
=\sum_{k=1}^n\langle m, \tilde{v}_k\rangle\log|\tilde{x}_k|^2
\]
holds, we obtain 
\[
 \mu^*\left(\max_{m\in P(\varphi_\nu)_\sigma}\log\prod_{j=1}^n|x_j|^{2\langle m, v_j\rangle}\right)
=\max_{m\in P(\varphi_\nu)_\sigma}\sum_{j=1}^n\langle m, \tilde{v}_j\rangle\log|\tilde{x}_j|^2. 
\]
As $\log|\tilde{x}_j|^2\leq 0$ holds for all $j$ on $K_{\tilde{\sigma}}$, the equation we desire can be rewritten as
\[
\mu^*\varphi_\nu\leq \log\prod_{j=1}^n|\tilde{x}_j|^{2\langle m_0, \tilde{v}_j\rangle}+M_{\varphi_\nu}, 
\]
where $m_0\in P(\varphi_\nu)_\sigma$ is such an element as in Lemma \ref{lem2}. 

Let $P(\varphi_\nu)_{\tilde{\sigma}}:=\frac{1}{\nu}P(\mu^*f^{(\nu)}_1, \mu^*f^{(\nu)}_2, \dots, \mu^*f^{(\nu)}_{N_\nu})_{(0, 0, \dots, 0, z_0)}$. 
According to Remark \ref{Pinv}, and since both $P(\varphi_\nu)_{\tilde{\sigma}}$ and $P(\varphi_\nu)_\sigma$ are generated by the same set, 
it turns out that $\mu^*f^{(\nu)}_j$ can be divided by the function $\prod_{k=1}^n(x_k)^{\langle \nu m_0, \tilde{v}_k\rangle}$ for all $j\in\{1, 2, \dots, N_\nu\}$. 
Denoting the quotient by $g^{(\nu)}_j$, 
the function $\mu^*\varphi_\nu-\log\prod_{j\in I}|\tilde{x}_j|^{2\langle m_0, \tilde{v}_j\rangle}$ can be rewritten as 
\[
\mu^*\varphi_\nu-\log\prod_{j=1}^n|\tilde{x}_j|^{2\langle m_0, \tilde{v}_j\rangle}
=\frac{1}{\nu}\log\sum_{j=1}^{N_\nu}|g^{(\nu)}_j|^2. 
\] 
Thus, this function is a plurisubharmonic function on $K_{\tilde{\sigma}}$, and it has the maximum value on $K_{\tilde{\sigma}}$, which we denote by $M_{\varphi_\nu, \tilde{\sigma}}$. 
Then, since 
\[
\mu^*\varphi_\nu\leq \log\prod_{j=1}^n|\tilde{x}_j|^{2\langle m_0, \tilde{v}_j\rangle}+M_{\varphi_\nu, \tilde{\sigma}}
\]
holds on $K_{\tilde{\sigma}}$. Therefore, it remains to prove that $M_{\varphi_\nu, \tilde{\sigma}}\leq M_{\varphi_\nu}$. 

Assume that the plurisubharmonic function $\mu^*\varphi_\nu-\log\prod_{j\in I}|\tilde{x}_j|^{2\langle m_0, \tilde{v}_j\rangle}$ has the maximum value 
at the point $((\tilde{x}_0)_1, (\tilde{x}_0)_2, \dots, (\tilde{x}_0)_n, z_0)\in K_{\tilde{\sigma}}$. 
We may assume 
$|(\tilde{x}_0)_j|=1$
for all $j$ after we change the point $((\tilde{x}_0)_1, (\tilde{x}_0)_2, \dots, (\tilde{x}_0)_n, z_0)\in K_{\tilde{\sigma}}$ if necessary. 
It is because, in the case when $|(\tilde{x}_0)_1|<1$ for example, by considering the plurisubharmonic function
\[
 \tilde{x}_1\mapsto\mu^*\varphi_\nu(\tilde{x}_1, (\tilde{x}_0)_2, (\tilde{x}_0)_3, \dots, (\tilde{x}_0)_n, z_0)
-\log\left(|\tilde{x}_1|^{2\langle m_0, \tilde{v}_1\rangle}\cdot\prod_{j=2}^n|(\tilde{x}_0)_j|^{2\langle m_0, \tilde{v}_j\rangle}\right)
\]
defined on $\{|\tilde{x}_1|<1\}$, the value of the function above must constantly be $M_{\varphi_\nu, \tilde{\sigma}}$. 

Then, we can calculate that 
\[
 M_{\varphi_\nu, \tilde{\sigma}}=
\mu^*\varphi_\nu((\tilde{x}_0)_1, (\tilde{x}_0)_2, \dots, (\tilde{x}_0)_n, z_0)
-\log\prod_{j=1}^n|(\tilde{x}_0)_j|^{2\langle m_0, \tilde{v}_j\rangle}
=\varphi_\nu(\mu((\tilde{x}_0)_1, (\tilde{x}_0)_2, \dots, (\tilde{x}_0)_n, z_0)). 
\]
Since $\mu((\tilde{x}_0)_1, (\tilde{x}_0)_2, \dots, (\tilde{x}_0)_n, z_0)\in K_\sigma$, the value is at most $M_{\varphi_\nu}$. 
\end{proof} 

\begin{proof}[Proof of Proposition \ref{main_theorem}]
Let us denote by $h$ the singular hermitian metric defined by $\{e^{-\psi_\sigma}\}_\sigma$, and by $h_{\infty}$ a smooth hermitian metric on $L$. 
Then, there exist upper semi-continuous functions $\varphi_{\rm min}'$ and $\psi'$ on $X$ such that 
\[
 h_{\rm min}=h_\infty e^{-\varphi_{\rm min}'},\ h=h_\infty e^{-\psi '}
\]
hold.  
Here, it is sufficient to prove that there exists a constant $C$ such that 
\[
\varphi_{\rm min}'\leq \psi '+C
\]
holds on $\pi^{-1}(\overline{U})\subset X$. \par
According to Lemma \ref{mainlem}, for each $\sigma\in\Sigma_{\rm max}$, there exists a constant $C_\sigma$ such that 
\[
 \varphi_{\rm min}'\leq \psi '+C_\sigma
\]
holds on the set $K_\sigma=\{(x_1, x_2, \dots, x_n, z)\in\mathbb{T}_N(\Sigma, \mathcal{L})|_{\overline{U}}\mid \forall j\in \{1, 2, \dots, n\}, |x_j|\leq 1\}$. 
Thus, according to Lemma \ref{1cover}, 
\[
\varphi_{\rm min}'\leq \psi '+C
\]
holds on $\pi^{-1}(\overline{U})\subset X$, where $C=\max_{\sigma\in\Sigma_{\rm max}}C_\sigma$. 
\end{proof}

\section{Properties related to the singularities of minimal singular metrics}

\subsection{Kiselman numbers and Lelong numbers of minimal singular metrics and Non-nef loci}

Let $X$ be a smooth projective variety and $L$ be a holomorphic line bundle over $X$. 
According to \cite[3.6]{B}, the next proposition follows. 

\begin{proposition}\label{NNef}
If $L$ is big, then the non-nef locus ${\rm NNef}(L)$ of $L$ can be written as 
\[
 {\rm NNef}(L)=\{x\in X\mid \nu(\varphi_{\rm min}, x)>0\}, 
\]
where $e^{-\varphi_{\rm min}}$ is a minimal singular metric on $L$. 
\end{proposition}

According to this proposition, 
we can specify the non-nef locus of a big line bundle 
by calculating the Lelong number of a minimal singular metric. 
It can be done, actually, in our setting. 

\begin{proposition}\label{Kiselman}
Let $X$ be the total space of a toric bundle $\mathbb{T}_N(\Sigma, \mathcal{L})$ over a complex torus and 
$L=\pi^*L_0\otimes\mathcal{O}_X(D_h)$ be a big line bundle over $X$, where $\Sigma$ is a smooth projective fan in a $n$-dimensional lattice $N$. 
The Kiselman number 
\[
\nu^K_{\zeta, w}(\varphi_{\rm min}, x_0)=\sup\left\{t\geq 0\left| \varphi_{\rm min}\leq t\log\sum_{j=1}^{n+d}|\zeta_j|^{2w_j}+O(1) \text{\ around\ }x_0 \right. \right\}
\]
associated to the coordinates system 
\[
\zeta=(\zeta_1, \zeta_2, \dots, \zeta_{n+d})=(x_1, x_2, \dots, x_n, z_1, z_2, \dots, z_d) 
\]
and $w=(w_j)\in\bigoplus_{j\in I}\mathbb{R}_{> 0}$ 
of a minimal singular metric $e^{-\varphi_{\rm min}}$ 
at a point \\
$x_0=((x_0)_1, (x_0)_2, \dots, (x_0)_n, z_0)\in\mathbb{T}_N(\sigma, \mathcal{L})$ 
(see \cite[Section 5.2]{BFJ} for the definition) 
can be calculated by using notations in the previous section that 
\[
\nu^K_{\zeta, w}(\varphi_{\rm min}, x_0)=\min_{m\in S(L_0, h)_\sigma}\left\langle m, \sum_{j\in I}\frac{v_j}{w_j}\right\rangle, 
\]
where we denote by $I$ the set $\{j\mid x_0^j=0\}$ and by $(x_1, x_2, \dots, x_n, z_1, z_2, \dots, z_d)$ 
the canonical coordinates system of $\mathbb{T}_N(\sigma, \mathcal{L})|_U$ associated to $N$-minimal generators 
$v_1, v_2, \dots, v_n$ of $\sigma$. 
Especially, the Lelong number at $x_0$ can be calculated that 
\[
\nu(\varphi_{\rm min}, x_0)
=\min_{m\in S(L_0, h)_\sigma}\sum_{j\in I}\langle m, v_j\rangle . 
\]
\end{proposition}

\begin{corollary}
Let $X, L$ be as that of the previous proposition. 
The following conditions 
are equivalent. 
\begin{enumerate}
\item $\varphi_{\rm min}(x_0)(=\psi_\sigma(x_0))=-\infty$. 
\item $\psi_\sigma$ is not continuous at $x_0$. 
\item $\nu(\varphi_{\rm min}, x_0)(=\nu(\psi_\sigma, x_0))>0$.  
\end{enumerate}
especially, 
\[
 \varphi_{\rm min}^{-1}(-\infty)={\rm Pole}(\varphi_{\rm min})
\]
holds, where we denote by ${\rm Pole}(\varphi_{\rm min})$ the set $\{x\in X\mid \nu(\varphi_{\rm min}, x)>0\}$. 
\end{corollary}

The next proposition is also obtained easily by Theorem \ref{main_theorem}. 

\begin{proposition}
Let $X, L$ be as that of Proposition \ref{Kiselman}. 
Then, ${\rm Pole}(\varphi_{\rm min})$ is a Zariski closed set. 
\end{proposition}

According to these argument, we obtain the following corollary. 

\begin{corollary}\label{NNef_is_Z-closed}
Let $X$ be the total space of a toric bundle $\mathbb{T}_N(\Sigma, \mathcal{L})$ over a complex torus and 
$L=\pi^*L_0\otimes\mathcal{O}_X(D_h)$ be a big line bundle over $X$, where $\Sigma$ is a smooth projective fan. 
Then, the set ${\rm NNef}(L)$ is a Zariski closed subset of $X$. 
\end{corollary}

\subsection{Multiplier ideal sheaves}
Let $\Sigma$ be a smooth projective fan of a $n$-dimensional lattice $N$. 
Fix $N$-minimal generators $v_1, v_2, \dots, v_n$ of $\sigma\in\Sigma_{\rm max}$. 
Let $(x_1, x_2, \dots, x_n, z)$ be the canonical coordinates system 
of $\mathbb{T}_N(\sigma, \mathcal{L})|_U$ associated to $v_1, v_2, \dots, v_n$, 
where $U$ is a sufficiently small open set in $V$. 
In this section, we consider the condition 
\[
f\in\mathcal{J}(h_{\rm min}^t)_{((x_0)_1, (x_0)_2, \dots, (x_0)_n, z_0)}, 
\] 
where $((x_0)_1, (x_0)_2, \dots, (x_0)_n, z_0)$ is a point of $\mathbb{T}_N(\sigma, \mathcal{L})|_U$, $f$ is an element of $\mathcal{O}_{X, ((x_0)_1, (x_0)_2, \dots, (x_0)_n, z)}\setminus\{0\}$, 
$t$ is a positive real number, 
and $h_{\rm min}$ is a minimal singular metric on $L$. 
In the following, we also denote by $\mathcal{J}(t\varphi_{\rm min})$ the multiplier ideal sheaf $\mathcal{J}(h_{\rm min}^t)$ 
by using the local weight function $\varphi_{\rm min}$ of the singular hermitian metric $h_{\rm min}$. 

Let $I:=\{j\in\{1, 2, \dots, n\}\mid (x_0)_j=0\}$. 
For this set $I$, let us denote the expansion appeared in Definition \ref{Pdef} by
\[
f(x_1, x_2, \dots, x_n, z)
=\sum_{m\in {\rm Pr}^I(\sigma^{\vee}\cap M)}\prod_{j\in I}(x_j)^{\langle m, v_j\rangle} A_{m} (x_{I^c}, z), 
\]
where the map ${\rm Pr}^I$ is the projection from $M_{\mathbb{R}}$ to ${\rm Span}_{\mathbb{R}}\{v^j\}_{j\in I}$. 
As the dual version of this map, we denote the projection from $N_{\mathbb{R}}$ to ${\rm Span}_{\mathbb{R}}\{v_j\}_{j\in I}$ by ${\rm Pr}_I$ in the following. 
Fix a set $A\subset P(f)_((x_0)_1, (x_0)_2, \dots, (x_0)_n, z_0)$ of lattice points such that 
\[
P(f)_{((x_0)_1, (x_0)_2, \dots, (x_0)_n, z_0)}=\overline{\overline{A}}
\]
holds. 

\begin{corollary}\label{multip_ideals_of_hmin}
The followings are equivalent. 
\begin{eqnarray}
&(1)&f\in\mathcal{J}(t\varphi_{\rm min})_{((x_0)_1, (x_0)_2, \dots, (x_0)_n, z_0)}. \nonumber \\
&(2)&\min_{m\in tS(L_0, h)_\sigma}\langle m, w\rangle
<\langle m_0+\sum_{j\in I}v^j, w\rangle \text{\ for\ all\ } m_0\in A\text{\ and\ }w\in{\rm Pr}_I(\sigma)\setminus\{0\}. \nonumber
\end{eqnarray}
\end{corollary}

Corollary \ref{multip_ideals_of_hmin} immediately follows from Theorem \ref{main_theorem} and 
the result of Guenancia \cite{G} referring to the way to compute the multiplier ideal sheaves associated to 
``toric plurisubharmonic functions", 
which can be regarded as a generalization of the famous Howald's result (\cite[Theorem 11]{H}) in algebraic setting. 

According to Corollary \ref{multip_ideals_of_hmin}, \cite[1.10, 1.11]{DEL}, and \cite[11.2.12 (ii)]{L2}, we obtain next corollary. 
\begin{corollary}\label{conti}
Let $X$ be the total space of a smooth projective toric bundle over a complex torus, 
$D$ a big divisor on $X$, 
and $e^{-\varphi_{\rm min}}$ be a minimal singular metric on the line bundle $\mathcal{O}_X(D)$. \\
$(1)$ If $f\in\mathcal{J}(t\varphi_{\rm min})_{x_0}$ at the point $x_0$, then 
$f\in\mathcal{J}((1+\varepsilon)t\varphi_{\rm min})_{x_0}$ holds for sufficiently small positive number $\varepsilon$ and any positive real number $t$.  
Especially, since the sheaf $\mathcal{J}(t\varphi_{\rm min})$ is coherent, it follows that 
\[
\mathcal{J}(t\varphi_{\rm min})=\mathcal{J}_+(t\varphi_{\rm min}). 
\]
 \\
$(2)$ Let $P$ be a nef big divisor on $X$, then 
\[
  H^j(X, \mathcal{O}_X(K_X+P+L)\otimes\mathcal{J}(\varphi_{\rm min}))=0
\]
holds for all $j>0$. 
\end{corollary}


\section{Some examples}
In this section, we will introduce three examples for $X$ and $L$ in the previous sections. 
We construct them as  $\mathbb{P}^2$-bundles over abelian surfaces, by following \cite[CHAPTER IV \S 2.6]{N} basically. 
In this section, we use notations in Example \ref{example3_2}. \par
As a preparation, we first recall a useful lemma to see $L$ is big. 

\begin{lemma}\label{cutkosky}
In the setting of Example \ref{example3_2}, $L$ is big if and only if there exists a triple $(a, b, c)$ of nonnegative integers such that
$L_0^a\otimes L_1^b\otimes L_2^c$ is ample line bundle over $V$. 
\end{lemma}

This lemma can be easily shown by applying the result known by Cutkosky (\cite[Lemma 2.3.2]{L1}) 
and the fact that the ample cones of complex tori coincide with these big cones. \par
Let $E$ be a sufficiently general smooth elliptic curve and $o$ be a point of $E$. 
For example, you can choose $\mathbb{C}/(\mathbb{Z}+(\pi+\sqrt{-1})\mathbb{Z})$ for $E$. 
Let 
\[V=E\times E. \] 
It is known that the rank of the Neron-Severi group ${\rm NS}(V)$ of $V$ is three and this group is generated by the following three classes (\cite[Chapter 1.5.B]{L1}). 
\begin{itemize}
\item $f_1=c_1(\mathcal{O}_V({F_1}))$ , \ \ where $F_1$ stands for the prime divisor $\{o\}\times E\subset V$. 
\item $f_2=c_1(\mathcal{O}_V({F_2}))$ , \ \ where $F_2$ stands for the prime divisor $E\times \{o\}\subset V$.
\item $\ \delta\,=c_1(\mathcal{O}_V({\Delta}))$\, , \ \ \ where $\Delta$ stands for the prime divisor $\{(x, y)\in E\times E\mid x=y\}$.
\end{itemize}
By using these three classes, the nef cone ${\rm Nef}(V)$ of $V$ can be written as 
\[{\rm Nef}(V)=\{af_1+bf_2+c\delta\mid a, b, c\in\mathbb{R},\ ab+bc+ca\geq 0,\ a+b+c\geq 0\}. \]
In order to obtain more useful expression of ${\rm Nef}(V)$, let us define the other basis of ${\rm NS}(V)\otimes\mathbb{R}$ by 
\[
l_1=\frac{1}{6}(f_1+f_2-2\delta),\ 
l_2=\frac{1}{6}(-\sqrt{3}f_1+\sqrt{3}f_2),\ 
{\rm and}\ l_3=\frac{1}{6}(f_1+f_2+\delta). 
\]
By using these classes, ${\rm Nef}(V)$ can be written as
\[{\rm Nef}(V)=\{al_1+bl_2+cl_3\mid c^2\geq a^2+b^2,\ c\geq 0\}. \]
This expression of ${\rm Nef}(V)$ makes it easy to judge the nef-ness of line bundles. 

\begin{example}\label{ex1}
The first example is an example which admits a Zariski decomposition after appropriate modifications. 
Let us fix two positive integers $u<v$. 
Let 
$L_0:=\mathcal{O}_V(-uF_1-uF_2-u\Delta), 
L_1:=\mathcal{O}_V((u+v)F_1+(u+v)F_2+(-2u+v)\Delta)$, and
$L_2:=\mathcal{O}_V((-u+v)F_1+(-u+v)F_2+(2u+v)\Delta)$. 
Then 
$c_1(L_0)=-6ul_3,\ c_1(L_1)=6(ul_1+vl_3)$, and $c_1(L_2)=6(-ul_1+vl_3)$ hold. 
These expressions make it clear that the line bundle $L_1\otimes L_2$ is ample 
and, according to Lemma \ref{cutkosky}, that $L$ is a big line bundle in this case. 

The set $\Box_{\rm Nef}(L_0, h)$ in this setting is rational polyhedral. 
More precisely, $\Box_{\rm Nef}(L_0, h)$ is the convex closure of the five points 
$e^1,\ e^2,\ \frac{u}{v}e^2,\ \frac{u}{2(u+v)}e^1+\frac{u}{2(u+v)}e^2,\ \frac{u}{v}e^1$
in $M_{\mathbb{R}}$. 
So, by applying Theorem \ref{main_theorem}, it immediately turns out that 
the weight of a minimal singular metric $\psi_{\sigma_j}$ satisfies
$\psi_{\sigma_j}\sim_{\rm sing}1$ 
at any points of $X$ except for the locus $\mathbb{P}(L_0)$, and 
\begin{eqnarray}\nonumber
\psi_{\sigma_1}(x_1, x_2, z)&\sim_{\rm sing}&\frac{u}{2v(u+v)}\log \max\{|x_1|^{2(2u+2v)}, |x_2|^{2(2u+2v)}, |x_1|^{2v}|x_2|^{2v}\} \\ \nonumber
&\sim_{\rm sing}&\frac{u}{2v(u+v)}\log \left(|x_1|^{2(2u+2v)}+|x_2|^{2(2u+2v)}+|x_1|^{2v}|x_2|^{2v}\right)
\end{eqnarray}
at a point $(0, 0, z_0)\in\mathbb{P}(L_0)$. 
Therefore, it follows that the non-nef locus ${\rm NNef}(L)$ is a Zariski closed subset $\mathbb{P}(L_0)$ of $X$. \par
According to \cite[2.5]{N}, the fact that $\Box_{\rm Nef}(L_0, h)$ is a rational polyhedral yields 
that $L$ admits a Zariski decomposition after appropriate proper modifications. 
Especially, when $u$ and $v$ can be written as
\[u=1, \ v=2n-2\]
for some integer $n>1$, 
$(X, L)$ is an example which admits a Zariski decomposition just after the $n$-time blow-up centered at the non-nef locus of the pull-back of $L$. 
It can be also checked out by using the above expression of the minimal singular metric on $L$. 
\par
According to the above expression of $\Box_{\rm Nef}(L_0, h)$, the result of Corollary \ref{multip_ideals_of_hmin} can be rewritten as follows. 
First, it is clear that $\mathcal{J}(h_{\rm min}^t)$ is trivial at any point in $X\setminus\mathbb{P}(L_0)$. 
Next, for a point $x_0\in\mathbb{P}(L_0)$, the stalk of $\mathcal{J}(h_{\rm min})_{x_0}$ of the multiplier ideal sheaf at $x_0$ is the ideal of $\mathcal{O}_{X, x_0}$ which is generated by the system of the polynomials
\[\{x_1^px_2^q\mid (p+1, q+1)\in {\rm Int}(S_t)\cap\mathbb{Z}^2\}, \] 
where we denote by ${\rm Int}(S_t)$ the interior of the set 
\[S_t=\{(\langle tm, e_1\rangle, \langle tm, e_2\rangle)\in\mathbb{R}^2\mid m\in S(L_0, h)_{\sigma_1}\}. \]
For the detail shape of $S_t$, see Figure \ref{example1}. 
\begin{figure}[htbp]
  \begin{center}
\unitlength 0.1in
\begin{picture}( 22.6000, 19.7000)(  2.3000,-21.1000)
%
{\color[named]{Black}{%
\special{pn 8}%
\special{pa 400 2010}%
\special{pa 2490 2010}%
\special{fp}%
\special{sh 1}%
\special{pa 2490 2010}%
\special{pa 2424 1990}%
\special{pa 2438 2010}%
\special{pa 2424 2030}%
\special{pa 2490 2010}%
\special{fp}%
}}%
%
{\color[named]{Black}{%
\special{pn 8}%
\special{pa 400 2010}%
\special{pa 400 140}%
\special{fp}%
\special{sh 1}%
\special{pa 400 140}%
\special{pa 380 208}%
\special{pa 400 194}%
\special{pa 420 208}%
\special{pa 400 140}%
\special{fp}%
}}%
%
{\color[named]{Black}{%
\special{pn 8}%
\special{pa 400 400}%
\special{pa 800 1600}%
\special{fp}%
}}%
%
{\color[named]{Black}{%
\special{pn 8}%
\special{pa 798 1600}%
\special{pa 2018 2010}%
\special{fp}%
}}%
\put(2.3000,-4.9000){\makebox(0,0)[lb]{$\frac{u}{v}$}}%
\put(19.9000,-22.4000){\makebox(0,0)[lb]{$\frac{u}{v}$}}%
\put(4.0000,-20.0000){\makebox(0,0)[lb]{$(\frac{u}{2(u+v)}, \frac{u}{2(u+v)})$}}%
%
{\color[named]{Black}{%
\special{pn 8}%
\special{pa 2470 350}%
\special{pa 1110 1710}%
\special{fp}%
\special{pa 2470 410}%
\special{pa 1160 1720}%
\special{fp}%
\special{pa 2470 470}%
\special{pa 1200 1740}%
\special{fp}%
\special{pa 2470 530}%
\special{pa 1250 1750}%
\special{fp}%
\special{pa 2470 590}%
\special{pa 1290 1770}%
\special{fp}%
\special{pa 2470 650}%
\special{pa 1340 1780}%
\special{fp}%
\special{pa 2470 710}%
\special{pa 1380 1800}%
\special{fp}%
\special{pa 2470 770}%
\special{pa 1430 1810}%
\special{fp}%
\special{pa 2470 830}%
\special{pa 1470 1830}%
\special{fp}%
\special{pa 2470 890}%
\special{pa 1520 1840}%
\special{fp}%
\special{pa 2470 950}%
\special{pa 1560 1860}%
\special{fp}%
\special{pa 2470 1010}%
\special{pa 1610 1870}%
\special{fp}%
\special{pa 2470 1070}%
\special{pa 1650 1890}%
\special{fp}%
\special{pa 2470 1130}%
\special{pa 1700 1900}%
\special{fp}%
\special{pa 2470 1190}%
\special{pa 1740 1920}%
\special{fp}%
\special{pa 2470 1250}%
\special{pa 1790 1930}%
\special{fp}%
\special{pa 2470 1310}%
\special{pa 1830 1950}%
\special{fp}%
\special{pa 2470 1370}%
\special{pa 1880 1960}%
\special{fp}%
\special{pa 2470 1430}%
\special{pa 1920 1980}%
\special{fp}%
\special{pa 2470 1490}%
\special{pa 1970 1990}%
\special{fp}%
\special{pa 2470 1550}%
\special{pa 2010 2010}%
\special{fp}%
\special{pa 2470 1610}%
\special{pa 2070 2010}%
\special{fp}%
\special{pa 2470 1670}%
\special{pa 2130 2010}%
\special{fp}%
\special{pa 2470 1730}%
\special{pa 2190 2010}%
\special{fp}%
\special{pa 2470 1790}%
\special{pa 2250 2010}%
\special{fp}%
\special{pa 2470 1850}%
\special{pa 2310 2010}%
\special{fp}%
\special{pa 2470 1910}%
\special{pa 2370 2010}%
\special{fp}%
\special{pa 2470 1970}%
\special{pa 2450 1990}%
\special{fp}%
\special{pa 2470 290}%
\special{pa 1070 1690}%
\special{fp}%
\special{pa 2470 230}%
\special{pa 1020 1680}%
\special{fp}%
}}%
%
{\color[named]{Black}{%
\special{pn 8}%
\special{pa 2470 170}%
\special{pa 980 1660}%
\special{fp}%
\special{pa 2440 140}%
\special{pa 930 1650}%
\special{fp}%
\special{pa 2380 140}%
\special{pa 890 1630}%
\special{fp}%
\special{pa 2320 140}%
\special{pa 840 1620}%
\special{fp}%
\special{pa 2260 140}%
\special{pa 800 1600}%
\special{fp}%
\special{pa 2200 140}%
\special{pa 790 1550}%
\special{fp}%
\special{pa 2140 140}%
\special{pa 770 1510}%
\special{fp}%
\special{pa 2080 140}%
\special{pa 760 1460}%
\special{fp}%
\special{pa 2020 140}%
\special{pa 740 1420}%
\special{fp}%
\special{pa 1960 140}%
\special{pa 730 1370}%
\special{fp}%
\special{pa 1900 140}%
\special{pa 710 1330}%
\special{fp}%
\special{pa 1840 140}%
\special{pa 700 1280}%
\special{fp}%
\special{pa 1780 140}%
\special{pa 680 1240}%
\special{fp}%
\special{pa 1720 140}%
\special{pa 670 1190}%
\special{fp}%
\special{pa 1660 140}%
\special{pa 650 1150}%
\special{fp}%
\special{pa 1600 140}%
\special{pa 640 1100}%
\special{fp}%
\special{pa 1540 140}%
\special{pa 620 1060}%
\special{fp}%
\special{pa 1480 140}%
\special{pa 610 1010}%
\special{fp}%
\special{pa 1420 140}%
\special{pa 590 970}%
\special{fp}%
\special{pa 1360 140}%
\special{pa 580 920}%
\special{fp}%
\special{pa 1300 140}%
\special{pa 560 880}%
\special{fp}%
\special{pa 1240 140}%
\special{pa 550 830}%
\special{fp}%
\special{pa 1180 140}%
\special{pa 530 790}%
\special{fp}%
\special{pa 1120 140}%
\special{pa 520 740}%
\special{fp}%
\special{pa 1060 140}%
\special{pa 500 700}%
\special{fp}%
\special{pa 1000 140}%
\special{pa 490 650}%
\special{fp}%
\special{pa 940 140}%
\special{pa 470 610}%
\special{fp}%
\special{pa 880 140}%
\special{pa 460 560}%
\special{fp}%
\special{pa 820 140}%
\special{pa 440 520}%
\special{fp}%
\special{pa 760 140}%
\special{pa 430 470}%
\special{fp}%
}}%
%
{\color[named]{Black}{%
\special{pn 8}%
\special{pa 700 140}%
\special{pa 410 430}%
\special{fp}%
\special{pa 640 140}%
\special{pa 400 380}%
\special{fp}%
\special{pa 580 140}%
\special{pa 400 320}%
\special{fp}%
\special{pa 520 140}%
\special{pa 400 260}%
\special{fp}%
\special{pa 420 180}%
\special{pa 400 200}%
\special{fp}%
\special{pa 460 140}%
\special{pa 420 180}%
\special{fp}%
}}%
\end{picture}%
    \caption{The shaded area of this figure represents the set $S_1$. 
The set $S_t$ is the set of points $p\in\mathbb{R}^2$ which satisfies $\frac{p}{t}\in S_1$. }
    \label{example1}
  \end{center}
\end{figure}
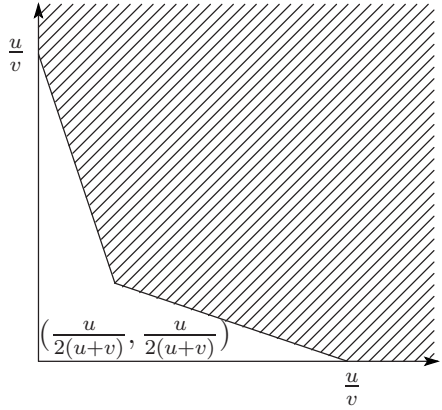

The set of the whole jumping numbers ${\rm Jump}(\psi_{\sigma_1}; x_0)$ at a point $x_0\in \mathbb{P}(L_0)$ can be written as 
${\rm Jump}(\psi_{\sigma_1}; x_0)=\left\{\left.2p+(p+q)\frac{v}{u}\right|p, q\in\mathbb{Z},\ 1\leq p\leq q\right\}$, 
and the singularity exponent $c_{x_0}(\psi_{\sigma_1})$, which is the least number in ${\rm Jump}(\psi_{\sigma_1}; x_0)$, satisfies
$c_{x_0}(\psi_{\sigma_1})=2\left(1+\frac{v}{u}\right)$. 
\end{example}

\begin{remark}\label{propertis_of_jumping_numbers_in_algebraic_cases}
In Example \ref{ex1}, the behavior of the multiplier ideal sheaf $\mathcal{J}(\psi_{\sigma_1})$ around a point of $\mathbb{P}(L_0)$ 
coincides with that of the (algebraic) multiplier ideal sheaf $\mathcal{J}(\mathfrak{a}^c)$, where $\mathfrak{a}$ is an ideal generated by $(x_1^{2(u+v)}, x_2^{2(u+v)}, x_1^vx_2^v)$ 
and $c$ is the rational number $\frac{u}{2v(u+v)}$. \par
This means that the analytic multiplier ideal sheaf $\mathcal{J}(\psi_{\sigma_1})_{x_0}$ has properties same as algebraic multiplier ideal sheaves. 
For example, it is known that, related to the algebraic multiplier ideal sheaf $\mathcal{J}(\mathfrak{a}^c)$, 
the set of the whole jumping numbers ${\rm Jump}(\mathfrak{a} ; x_0)$ is a discrete subset of the set of rational numbers $\mathbb{Q}$, and has the property so-called ``periodicity" in a sufficiently big parts of this set (see \cite[1.12]{ELSV} for details). 
Indeed, it can be easily checked that ${\rm Jump}(\psi_{\sigma_1}; x_0)$ is a discrete subset of $\mathbb{Q}$, and has a ``period" $c^{-1}=2v(1+\frac{v}{u})$. 
\end{remark}

\begin{example}\label{Nakayama}
Second example is the example found out by Nakayama (\cite{N}), 
which admits no Zariski decomposition even after modifications. \par
Let us fix an integer $a>1$ and set
$
L_0:=\mathcal{O}_V(2F_1-4F_2+2\Delta), 
L_1:=\mathcal{O}_V((a-1)F_1+(a-1)F_2+(a+2)\Delta)$
, and
$L_2:=\mathcal{O}_V((a+3)F_1+(a-3)F_2+a\Delta)$. 
Then 
$c_1(L_0)=-6(l_1+\sqrt{3}l_2),\ c_1(L_1)=6(-l_1+al_3)$, and $c_1(L_2)=6(-\sqrt{3}l_2+al_3)$ hold. 
By these expressions, it turns out that the line bundles $L_1$ and $L_2$ are ample 
and, according to Lemma \ref{cutkosky}, that $L$ is also a big line bundle in this case. 
For this example, see Section 1. 
\end{example}

\begin{example}
Finally, we introduce an example which can be proved that admits no Zariski decomposition even after modifications 
in the almost same way to the case of previous Nakayama example, 
however whose minimal singular metric can be expressed more easily. 

Let
$
L_0:=\mathcal{O}_V(4F_1+4F_2+\Delta), 
L_1:=\mathcal{O}_V$
, and 
$L_2:=\mathcal{O}_V(-F_1+9F_2+\Delta)$. 
Then 
$c_1(L_0)=6(l_1+3l_3),\ c_1(L_1)=0$, and $c_1(L_2)=6l_1+10\sqrt{3}l_2+18l_3$ hold. 
By this expression, it turns out that the line bundle $L_0$ is ample 
and, from Lemma \ref{cutkosky}, that $L$ is also a big line bundle in this case. 

The set $\Box_{\rm Nef}(L_0, h)$ in this setting is not rational, but is polyhedral. 
More precisely, $\Box_{\rm Nef}(L_0, h)$ is the convex closure of the three points 
$0,\ e^1$, and $\frac{2\sqrt{6}}{5}e^2$
in $M_{\mathbb{R}}$. 
So, applying theorem \ref{main_theorem}, it  immediately turns out that 
the weight of a minimal singular metric $\psi_{\sigma_j}$ satisfies $\psi_{\sigma_j}\sim_{\rm sing}1$ 
at any points of $X$ except for the locus $\mathbb{P}(L_2)$, and 
\begin{eqnarray}\nonumber
\psi_{\sigma_3}(x_1, x_2, z)&\sim_{\rm sing}&\log \max\{|x_0|^{2\alpha}, |x_1|^{2}\} \\ \nonumber
&\sim_{\rm sing}&\log \left(|x_0|^{2\alpha}+|x_1|^{2}\right)
\end{eqnarray}
at a point $(0, 0, z_0)\in\mathbb{P}(L_2)$, where we denote by $\alpha$ the positive irrational number $1-\frac{2\sqrt{6}}{5}$. 

According to the above expression of $\Box_{\rm Nef}(L_0, h)$, the result of Corollary \ref{multip_ideals_of_hmin} can be rewritten as follows. 
First, it is clear that $\mathcal{J}(h_{\rm min}^t)$ is trivial at any point in $X\setminus\mathbb{P}(L_2)$. 
Next, for a point $x_0\in\mathbb{P}(L_2)$, the stalk $\mathcal{J}(h_{\rm min})_{x_0}$ of the multiplier ideal sheaf at $x_0$ is the ideal of $\mathcal{O}_{X, x_0}$ which is generated by the polynomials
\[\{x_1^px_2^q\mid (p+1, q+1)\in {\rm Int}(S_t)\cap\mathbb{Z}^2\}, \] 
where we denote by $S_t$ the set $\{(\langle tm, e_1\rangle, \langle tm, e_2\rangle)\in\mathbb{R}^2\mid m\in S(L_0, h)_{\sigma_3}\}$. For the detail shape of $S_t$ in this case, see Figure \ref{example3}. 
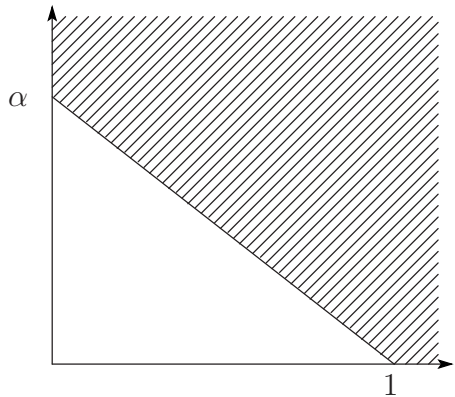
\begin{figure}[htbp]
  \begin{center}
\unitlength 0.1in
\begin{picture}( 23.2000, 19.0000)(  1.7000,-20.4000)
%
{\color[named]{Black}{%
\special{pn 8}%
\special{pa 400 2010}%
\special{pa 2490 2010}%
\special{fp}%
\special{sh 1}%
\special{pa 2490 2010}%
\special{pa 2424 1990}%
\special{pa 2438 2010}%
\special{pa 2424 2030}%
\special{pa 2490 2010}%
\special{fp}%
}}%
%
{\color[named]{Black}{%
\special{pn 8}%
\special{pa 400 2010}%
\special{pa 400 140}%
\special{fp}%
\special{sh 1}%
\special{pa 400 140}%
\special{pa 380 208}%
\special{pa 400 194}%
\special{pa 420 208}%
\special{pa 400 140}%
\special{fp}%
}}%
%
{\color[named]{Black}{%
\special{pn 8}%
\special{pa 400 610}%
\special{pa 2190 2010}%
\special{fp}%
}}%
%
{\color[named]{Black}{%
\special{pn 8}%
\special{pa 2420 700}%
\special{pa 1590 1530}%
\special{fp}%
\special{pa 2420 760}%
\special{pa 1620 1560}%
\special{fp}%
\special{pa 2420 820}%
\special{pa 1650 1590}%
\special{fp}%
\special{pa 2420 880}%
\special{pa 1690 1610}%
\special{fp}%
\special{pa 2420 940}%
\special{pa 1720 1640}%
\special{fp}%
\special{pa 2420 1000}%
\special{pa 1750 1670}%
\special{fp}%
\special{pa 2420 1060}%
\special{pa 1790 1690}%
\special{fp}%
\special{pa 2420 1120}%
\special{pa 1820 1720}%
\special{fp}%
\special{pa 2420 1180}%
\special{pa 1860 1740}%
\special{fp}%
\special{pa 2420 1240}%
\special{pa 1890 1770}%
\special{fp}%
\special{pa 2420 1300}%
\special{pa 1920 1800}%
\special{fp}%
\special{pa 2420 1360}%
\special{pa 1960 1820}%
\special{fp}%
\special{pa 2420 1420}%
\special{pa 1990 1850}%
\special{fp}%
\special{pa 2420 1480}%
\special{pa 2020 1880}%
\special{fp}%
\special{pa 2420 1540}%
\special{pa 2060 1900}%
\special{fp}%
\special{pa 2420 1600}%
\special{pa 2090 1930}%
\special{fp}%
\special{pa 2420 1660}%
\special{pa 2120 1960}%
\special{fp}%
\special{pa 2420 1720}%
\special{pa 2160 1980}%
\special{fp}%
\special{pa 2420 1780}%
\special{pa 2190 2010}%
\special{fp}%
\special{pa 2420 1840}%
\special{pa 2250 2010}%
\special{fp}%
\special{pa 2420 1900}%
\special{pa 2310 2010}%
\special{fp}%
\special{pa 2420 1960}%
\special{pa 2370 2010}%
\special{fp}%
\special{pa 2420 640}%
\special{pa 1550 1510}%
\special{fp}%
\special{pa 2420 580}%
\special{pa 1520 1480}%
\special{fp}%
\special{pa 2420 520}%
\special{pa 1490 1450}%
\special{fp}%
\special{pa 2420 460}%
\special{pa 1450 1430}%
\special{fp}%
\special{pa 2420 400}%
\special{pa 1420 1400}%
\special{fp}%
\special{pa 2420 340}%
\special{pa 1380 1380}%
\special{fp}%
\special{pa 2420 280}%
\special{pa 1350 1350}%
\special{fp}%
\special{pa 2420 220}%
\special{pa 1320 1320}%
\special{fp}%
}}%
%
{\color[named]{Black}{%
\special{pn 8}%
\special{pa 2390 190}%
\special{pa 1280 1300}%
\special{fp}%
\special{pa 2330 190}%
\special{pa 1250 1270}%
\special{fp}%
\special{pa 2270 190}%
\special{pa 1220 1240}%
\special{fp}%
\special{pa 2210 190}%
\special{pa 1180 1220}%
\special{fp}%
\special{pa 2150 190}%
\special{pa 1150 1190}%
\special{fp}%
\special{pa 2090 190}%
\special{pa 1110 1170}%
\special{fp}%
\special{pa 2030 190}%
\special{pa 1080 1140}%
\special{fp}%
\special{pa 1970 190}%
\special{pa 1050 1110}%
\special{fp}%
\special{pa 1910 190}%
\special{pa 1010 1090}%
\special{fp}%
\special{pa 1850 190}%
\special{pa 980 1060}%
\special{fp}%
\special{pa 1790 190}%
\special{pa 950 1030}%
\special{fp}%
\special{pa 1730 190}%
\special{pa 910 1010}%
\special{fp}%
\special{pa 1670 190}%
\special{pa 880 980}%
\special{fp}%
\special{pa 1610 190}%
\special{pa 850 950}%
\special{fp}%
\special{pa 1550 190}%
\special{pa 810 930}%
\special{fp}%
\special{pa 1490 190}%
\special{pa 780 900}%
\special{fp}%
\special{pa 1430 190}%
\special{pa 740 880}%
\special{fp}%
\special{pa 1370 190}%
\special{pa 710 850}%
\special{fp}%
\special{pa 1310 190}%
\special{pa 680 820}%
\special{fp}%
\special{pa 1250 190}%
\special{pa 640 800}%
\special{fp}%
\special{pa 1190 190}%
\special{pa 610 770}%
\special{fp}%
\special{pa 1130 190}%
\special{pa 580 740}%
\special{fp}%
\special{pa 1070 190}%
\special{pa 540 720}%
\special{fp}%
\special{pa 1010 190}%
\special{pa 510 690}%
\special{fp}%
\special{pa 950 190}%
\special{pa 480 660}%
\special{fp}%
\special{pa 890 190}%
\special{pa 440 640}%
\special{fp}%
\special{pa 830 190}%
\special{pa 410 610}%
\special{fp}%
\special{pa 770 190}%
\special{pa 400 560}%
\special{fp}%
\special{pa 710 190}%
\special{pa 400 500}%
\special{fp}%
\special{pa 650 190}%
\special{pa 400 440}%
\special{fp}%
}}%
%
{\color[named]{Black}{%
\special{pn 8}%
\special{pa 590 190}%
\special{pa 400 380}%
\special{fp}%
\special{pa 530 190}%
\special{pa 400 320}%
\special{fp}%
\special{pa 470 190}%
\special{pa 400 260}%
\special{fp}%
}}%
\put(1.7000,-6.6000){\makebox(0,0)[lb]{$\alpha$}}%
\put(21.3000,-21.7000){\makebox(0,0)[lb]{$1$}}%
\end{picture}%
    \caption{The shaded area of this figure represents the set $S_1$. 
The set $S_t$ is the set of points $p\in\mathbb{R}^2$ which satisfies $\frac{p}{t}\in S_1$. }
    \label{example3}
  \end{center}
\end{figure}
\par
Let $x_0$ be a point in $\mathbb{P}(L_2)$. 
In this case, ${\rm Jump}(\psi_{\sigma_3}; x_0)$ can be calculated that 
${\rm Jump}(\psi_{\sigma_3}; x_0)=\mathbb{Z}_{>0}\oplus\frac{1}{\alpha}\cdot\mathbb{Z}_{>0}$, 
and the singularity exponent can be calculated that 
$c_{x_0}(\psi_{\sigma_1})=1+\frac{1}{\alpha}$, 
which is not rational, too. 
It can easily be proved by using (\cite[2.11]{N}) that $L$ admits no Zariski decomposition even after modifications in this settings. 

\end{example}


\end{document}